\documentclass{amsart}

\usepackage[all]{xy}
\xyoption{poly}
\xyoption{arc}
\xyoption{curve}

\CompileMatrices

\begin{document}

\renewcommand{\theequation}{\thesection.\arabic{equation}}
\newcommand{\nc}{\newcommand}

\nc{\pr}{\noindent{\em Proof. }} \nc{\g}{\mathfrak g}
\nc{\n}{\mathfrak n} \nc{\opn}{\overline{\n}}\nc{\h}{\mathfrak h}
\renewcommand{\b}{\mathfrak b}
\nc{\Ug}{U(\g)} \nc{\Uh}{U(\h)} \nc{\Un}{U(\n)}
\nc{\Uopn}{U(\opn)}\nc{\Ub}{U(\b)} \nc{\p}{\mathfrak p}
\renewcommand{\l}{\mathfrak l}
\nc{\z}{\mathfrak z} \renewcommand{\h}{\mathfrak h}
\nc{\m}{\mathfrak m}
\renewcommand{\k}{\mathfrak k}
\nc{\opk}{\overline{\k}}
\nc{\opb}{\overline{\b}}
\nc{\e}{{\epsilon}}
\nc{\ke}{{\bf k}_\e}
\nc{\Hk}{{\rm Hk}^{\gr}(A,A_0,\e )}
\nc{\gr}{\bullet}
\nc{\ra}{\rightarrow}
\nc{\Alm}{A-{\rm mod}}
\nc{\DAl}{{D}^-(A)}
\nc{\HA}{{\rm Hom}_A}

\newtheorem{theorem}{Theorem}{}
\newtheorem{lemma}[theorem]{Lemma}{}
\newtheorem{corollary}[theorem]{Corollary}{}
\newtheorem{conjecture}[theorem]{Conjecture}{}
\newtheorem{proposition}[theorem]{Proposition}{}
\newtheorem{axiom}{Axiom}{}
\newtheorem{remark}[theorem]{Remark}{}
\newtheorem{example}{Example}{}
\newtheorem{exercise}{Exercise}{}
\newtheorem{definition}{Definition}{}

\renewcommand{\thetheorem}{\thesection.\arabic{theorem}}

\renewcommand{\thelemma}{\thesection.\arabic{lemma}}

\renewcommand{\theproposition}{\thesection.\arabic{proposition}}

\renewcommand{\thecorollary}{\thesection.\arabic{corollary}}

\renewcommand{\theremark}{\thesection.\arabic{remark}}

\renewcommand{\thedefinition}{\thesection.\arabic{definition}}

\title{Representations of quantum groups at roots of unity, Whittaker vectors and q-W algebras}

\author{A. Sevostyanov}

\address{Institute of Pure and Applied Mathematics,
University of Aberdeen \\ Aberdeen AB24 3UE, United Kingdom \\e-mail: a.sevastyanov@abdn.ac.uk}

\begin{abstract}
Let $U_\varepsilon(\g)$ be the standard simply connected version of the Drinfeld--Jumbo quantum group at an odd m-th root of unity $\varepsilon$. The center of $U_\varepsilon(\mathfrak g)$ contains a huge commutative subalgebra isomorphic to the algebra $Z_G$ of regular functions on (a finite covering of a big cell in) a complex connected, simply connected algebraic group $G$ with Lie algebra $\mathfrak g$.
Let $V$ be a finite--dimensional representation of $U_\varepsilon(\mathfrak g)$ on which $Z_G$ acts according to a non--trivial character $\eta_g$ given by evaluation of regular functions at $g\in G$. Then $V$ is a representation of the finite--dimensional algebra $U_{\eta_g}=U_\varepsilon(\mathfrak g)/U_\varepsilon(\mathfrak g){\rm Ker}~\eta_g$.
We show that in this case, under certain restrictions on $m$, $U_{\eta_g}$ contains a subalgebra $U_{\eta_g}(\m_-)$ of dimension $m^{{\frac{1}{2}}{\rm dim}~\mathcal{O}}$, where $\mathcal{O}$ is the conjugacy class of $g$, and $U_{\eta_g}(\m_-)$ has a one--dimensional representation $\mathbb{C}_{\chi_g}$. We also prove that if $V$ is not trivial then the space of Whittaker vectors ${\rm Hom}_{U_{\eta_g}(\m_-)}(\mathbb{C}_{\chi_g},V)$ is not trivial and the algebra $W_{\eta_g}={\rm End}_{U_{\eta_g}}(U_{\eta_g}\otimes_{U_{\eta_g}(\m_-)}\mathbb{C}_{\chi_g})$ naturally acts on it which gives rise to a Schur--type duality between representations of the algebra $U_{\eta_g}$ and of the algebra $W_{\eta_g}$ called a q-W algebra.
\end{abstract}


\maketitle

\pagestyle{myheadings}
\markboth{A. SEVOSTYANOV}{REPRESENTATIONS OF QUANTUM GROUPS AT ROOTS OF UNITY AND Q-W ALGEBRAS}


\section{Introduction}

\setcounter{equation}{0}
\setcounter{theorem}{0}

Let $\frak g'$ be the Lie algebra of a semisimple algebraic group $G'$ over an algebraically closed field $\bf k$ of characteristic $p>0$. Let $x\mapsto x^{[p]}$ be the $p$-th power map of $\frak g'$ into itself. The structure of the enveloping algebra of $\frak g'$ is quite different from the zero characteristic case. Namely, the elements $x^{p}-x^{[p]},~x\in \frak g'$ are central.
For any linear form $\theta$ on $\frak g'$, let $U_{\theta}$ be the
quotient of
the enveloping algebra of $\frak g'$ by the ideal generated by the central elements
$x^{p}-x^{[p]}-\theta (x)^{p}$ with $x\in \frak g'$.  Then $U_{\theta }$ is a finite--dimensional algebra. Kac and Weisfeiler proved that any simple $\frak g'$-module can be regarded as a module over $U_{\theta }$ for a unique $\theta$ as above (this explains why all simple $\frak g'$--modules are finite--dimensional). 

One can identify $\theta $ with an element of $\frak g'$ via the Killing form and reduce the general situation to the case of nilpotent $\theta$. In that case, among other things, Premet defines in \cite{Pr1,Pr} a subalgebra $U_\theta(\m_\theta)\subset U_{\theta }$ generated by a Lie subalgebra $\m_\theta\subset \g'$ such that $U_\theta(\m_\theta)$ has dimension $p^{\frac d2}$, where d is the dimension of the coadjoint orbit of $\theta$, and is equipped with a one--dimensional representation ${\bf k}_{\chi_\theta}$, $\chi_\theta: U_\theta(\m_\theta)\rightarrow {\bf k}$. Thus for every $U_{\theta }$--module $V$ the algebra ${\rm End}_{U_{\theta }}(U_{\theta }\otimes_{U_\theta(\m_\theta)}{\bf k}_{\chi_\theta})$ naturally acts on the space ${\rm Hom}_{U_{\theta }}(U_{\theta }\otimes_{U_\theta(\m_\theta)}{\bf k}_{\chi_\theta},V)$ called the space of Whittaker vectors in $V$. This leads to a Schur--type duality between representations of the algebras $U_{\theta }$ and ${\rm End}_{U_{\theta }}(U_{\theta }\otimes_{U_\theta(\m_\theta)}{\bf k}_{\chi_\theta})$, the latter being called a W--algebra. Tensoring $V$ with the one--dimensional representation ${\bf k}_{-\chi_\theta}$ of $U_\theta(\m_\theta)$ one can define the structure of a $U_0(\m_\theta)$--module (here $U_0(\m_\theta)$ is the subalgebra generated by $\m_\theta$ in $U_0$ which corresponds to $\theta=0$) on $V$, which makes it a $U_{\theta }-U_0(\m_\theta)$--bimodule, and the two actions satisfy certain compatibility conditions. This brings the representation theory of the algebra $U_{\theta }$ into the context of equivariant representation theory showing also some similarity with the theory of generalized Gelfand--Graev representations (see \cite{Ka}).

Another important example of finite--dimensional algebras is related to the theory of quantum groups at roots of unity. Let $\mathfrak g$ be a complex finite--dimensional semisimple Lie algebra.
A remarkable property of the standard Drinfeld-Jimbo quantum group $U_\varepsilon(\mathfrak g)$ associated to $\mathfrak g$, where $\varepsilon$ is a primitive  $m$-th root of unity, is that its center  contains a huge commutative subalgebra isomorphic to the algebra $Z_G$ of regular functions on (a finite covering of a big cell in) a complex algebraic group $G$ with Lie algebra $\mathfrak g$ (see \cite{DK, DKP1}). In this paper we consider the simply connected version of $U_\varepsilon(\mathfrak g)$ and the case when $m$ is odd. In that case $G$ is the connected, simply connected algebraic group corresponding to $\g$.

Consider finite--dimensional representations of $U_\varepsilon(\mathfrak g)$, on which $Z_G$ acts according to non--trivial characters $\eta_g$ given by evaluation of regular functions at various points $g\in G$. Note that all irreducible representations of $U_\varepsilon(\mathfrak g)$ are of that kind, and every such representation is a representation of the algebra $U_{\eta_g}=U_\varepsilon(\mathfrak g)/U_\varepsilon(\mathfrak g){\rm Ker}~\eta_g$ for some $\eta_g$. 

In this paper we construct certain subalgebras $U_{\eta_g}(\m_-)$ in $U_{\eta_g}$ which have properties similar to those of the subalgebras $U_\theta(\m_\theta)\subset U_{\theta }$. In particular, there is a Schur--type duality between representations of the algebras $U_{\eta_g}$ and of certain quantum group versions of W--algebras called q-W--algebras. $U_{\eta_g}$--modules also naturally become $U_{\eta_g}-U_{\eta_1}(\m_-)$--bimodules.  We show that for such modules the corresponding spaces of Whittaker vectors are always non--trivial which can be regarded as an analogue of the Engel theorem for quantum groups at roots of unity.

It turns out that the definition of the subalgebras $U_{\eta_g}(\m_-)$ is related to the existence of some special transversal slices $\Sigma_s$ to the set of conjugacy classes in $G$. These slices $\Sigma_s$ associated to (conjugacy classes of) elements $s$ in the Weyl group of $\g$ were introduced by the author in \cite{S6}. The slices $\Sigma_s$ play the role of the Slodowy slices in algebraic group theory. In the particular case of elliptic Weyl group elements these slices were also introduced later by He and Lusztig in paper \cite{Lus} within a different framework.

A remarkable property of a slice $\Sigma_s$ is that if $g$ is conjugate to an element in $\Sigma_s$ then $U_{\eta_g}$ has a subalgebra of dimension $m^{\frac{1}{2}{\rm codim}~\Sigma_s}$ with a non--trivial character. If $g\in \Sigma_s$ (in fact $g$ may belong to a larger variety) then the corresponding subalgebra $U_{\eta_g}(\m_-)$ can be explicitly described in terms of quantum group analogues of root vectors. There are also analogues of subalgebras $U_{\eta_g}(\m_-)$ in $U_q(\mathfrak g)$ in case of generic $q$ (see \cite{S9}).

Q-W--algebras can be regarded as noncommutative deformations of truncated versions of the algebras of regular functions on the slices $\Sigma_s$. In case of generic $\varepsilon$ q-W algebras were introduced and studied in \cite{S9}.

In \cite{S10}, Theorem 5.2 it is shown that for every conjugacy class $\mathcal{O}$ in $G$ one can find a transversal slice $\Sigma_s$ such that $\mathcal{O}$ intersects $\Sigma_s$ and ${\rm dim}~\mathcal{O}={\rm codim}~\Sigma_s$. Thus in this case for the corresponding algebra $U_{\eta_g}(\m_-)$ we have ${\rm dim}~U_{\eta_g}(\m_-)=m^{{\frac{1}{2}}{\rm dim}~\mathcal{O}}$, where $\mathcal{O}$ is the conjugacy class of $g$.

There is, however, a major difference between Lie algebras and quantum groups: in case of Lie algebras $\g'$ over fields of prime characteristic the algebras $U_\theta(\m_\theta)$ are local while in the quantum group case the algebras ${U}_{\eta_g}(\m_-)$, which play the role of $U_\theta(\m_\theta)$, are not local. In particular, each ${U}_{\eta_g}(\m_-)$ may have several one--dimensional representations.

The author is grateful to the referee for careful reading of the text.


\section{Notation}\label{notation}

\setcounter{equation}{0}
\setcounter{theorem}{0}

Fix the notation used throughout the text.
Let $G$ be a
simply connected finite--dimensional complex simple Lie group, $
{\frak g}$ its Lie algebra. Fix a Cartan subalgebra ${\frak h}\subset {\frak
g}$ and let $\Delta $ be the set of roots of $\left( {\frak g},{\frak h}
\right)$.  Let $\alpha_i,~i=1,\ldots, l,~~l=rank({\frak g})$ be a system of
simple roots, $\Delta_+=\{ \beta_1, \ldots ,\beta_N \}$
the set of positive roots.
Let $H_1,\ldots ,H_l$ be the set of simple root generators of $\frak h$.

Let $a_{ij}$ be the corresponding Cartan matrix,
and let $d_1,\ldots , d_l$, $d_i\in \{1,2,3\}$, $i=1,\ldots , l$ be coprime positive integers such that the matrix
$b_{ij}=d_ia_{ij}$ is symmetric. There exists a unique non--degenerate invariant
symmetric bilinear form $\left( ,\right) $ on ${\frak g}$ such that
$(H_i , H_j)=d_j^{-1}a_{ij}$. It induces an isomorphism of vector spaces
${\frak h}\simeq {\frak h}^*$ under which $\alpha_i \in {\frak h}^*$ corresponds
to $d_iH_i \in {\frak h}$. We denote by $\alpha^\vee$ the element of $\frak h$ that
corresponds to $\alpha \in {\frak h}^*$ under this isomorphism.
The induced bilinear form on ${\frak h}^*$ is given by
$(\alpha_i , \alpha_j)=b_{ij}$.

Let $W$ be the Weyl group of the root system $\Delta$. $W$ is the subgroup of $GL({\frak h})$
generated by the fundamental reflections $s_1,\ldots ,s_l$,
$$
s_i(h)=h-\alpha_i(h)H_i,~~h\in{\frak h}.
$$
The action of $W$ preserves the bilinear form $(,)$ on $\frak h$.

For any root $\alpha\in \Delta$ we also denote by $s_\alpha$ the corresponding reflection.

For every element $w\in W$ one can introduce the set $\Delta_w=\{\alpha \in \Delta_+: w(\alpha)\in -\Delta_+\}$, and the number of the elements in the set $\Delta_w$ is equal to the length $l(w)$ of the element $w$ with respect to the system $\Gamma$ of simple roots in $\Delta_+$.

Let $\b_+$ be the Borel subalgebra corresponding to $\Delta_+$, $\b_-$ the opposite Borel subalgebra, $\n_\pm$ their nilradicals.
Denote by $H,N_+, N_-,B_+,B_-$ the maximal torus, the maximal unipotent subgroups and the Borel subgroups of $G$ which correspond to the Lie subalgebras ${\frak h},{{\frak n}_+},{\frak n}_-,{\frak b}_+$ and ${\frak b}_-,$ respectively.

We identify $\frak g$ and its dual by means of the canonical invariant bilinear form.
Then the coadjoint
action of $G$ on ${\frak g}^*$ is naturally identified with the adjoint one. 

Let ${\frak g}_\beta$ be the root subspace corresponding to a root $\beta \in \Delta$,
${\frak g}_\beta=\{ x\in {\frak g}| [h,x]=\beta(h)x \mbox{ for every }h\in {\frak h}\}$.
${\frak g}_\beta\subset {\frak g}$ is a one--dimensional subspace.
It is well known that for $\alpha\neq -\beta$ the root subspaces ${\frak g}_\alpha$ and ${\frak g}_\beta$ are orthogonal with respect
to the canonical invariant bilinear form. Moreover ${\frak g}_\alpha$ and ${\frak g}_{-\alpha}$
are non--degenerately paired by this form.

Root vectors $X_{\alpha}\in {\frak g}_\alpha$ satisfy the following relations:
$$
[X_\alpha,X_{-\alpha}]=(X_\alpha,X_{-\alpha})\alpha^\vee.
$$

Note also that in this paper we denote by $\mathbb{N}$ the set of nonnegative integer numbers, $\mathbb{N}=\{0,1,\ldots \}$.


\section{Quantum groups}

\setcounter{equation}{0}
\setcounter{theorem}{0}

In this paper we shall consider some specializations of the standard Drinfeld-Jimbo quantum group $U_h({\frak g})$ defined over the ring of formal power series ${\Bbb C}[[h]]$, where $h$ is an indeterminate.
We follow the notation of \cite{ChP}.

Let $V$ be a ${\Bbb C}[[h]]$--module equipped with the $h$--adic
topology. This topology is characterized by requiring that
$\{ h^nV ~|~n\geq 0\}$ is a base of the neighborhoods of $0$ in $V$, and that translations
in $V$ are continuous.

A topological Hopf algebra over ${\Bbb C}[[h]]$ is a complete ${\Bbb C}[[h]]$--module 
equipped with a structure of ${\Bbb C}[[h]]$--Hopf algebra (see \cite{ChP}, Definition 4.3.1),
the algebraic tensor products entering the axioms of the Hopf algebra are replaced by their
completions in the $h$--adic topology.

The standard quantum group $U_h({\frak g})$ associated to a complex finite--dimensional simple Lie algebra
$\frak g$ is a topological Hopf algebra over ${\Bbb C}[[h]]$ topologically generated by elements
$H_i,~X_i^+,~X_i^-,~i=1,\ldots ,l$, subject to the following defining relations:
$$
\begin{array}{l}
[H_i,H_j]=0,~~ [H_i,X_j^\pm]=\pm a_{ij}X_j^\pm, ~~X_i^+X_j^- -X_j^-X_i^+ = \delta _{i,j}{K_i -K_i^{-1} \over q_i -q_i^{-1}},\\
\\
\sum_{r=0}^{1-a_{ij}}(-1)^r
\left[ \begin{array}{c} 1-a_{ij} \\ r \end{array} \right]_{q_i}
(X_i^\pm )^{1-a_{ij}-r}X_j^\pm(X_i^\pm)^r =0 ,~ i \neq j ,
\end{array}
$$
where
$$
K_i=e^{d_ihH_i},~~e^h=q,~~q_i=q^{d_i}=e^{d_ih},
$$
$$
\left[ \begin{array}{c} m \\ n \end{array} \right]_q={[m]_q! \over [n]_q![n-m]_q!} ,~
[n]_q!=[n]_q\ldots [1]_q ,~ [n]_q={q^n - q^{-n} \over q-q^{-1} },
$$
with comultiplication defined by
$$
\Delta_h(H_i)=H_i\otimes 1+1\otimes H_i,~~
\Delta_h(X_i^+)=X_i^+\otimes K_i^{-1}+1\otimes X_i^+,~~\Delta_h(X_i^-)=X_i^-\otimes 1 +K_i\otimes X_i^-,
$$
antipode defined by
$$
S_h(H_i)=-H_i,~~S_h(X_i^+)=-X_i^+K_i,~~S_h(X_i^-)=-K_i^{-1}X_i^-,
$$
and counit defined by
$$
\varepsilon_h(H_i)=\varepsilon_h(X_i^\pm)=0.
$$

We shall also use the weight--type generators
$$
Y_i=\sum_{j=1}^l d_i(a^{-1})_{ij}H_j.
$$
Let $L_i^{\pm 1}=e^{\pm hY_i}$.

Now we shall explicitly describe a linear basis for $U_h({\frak g})$.
First following \cite{ChP} we recall the construction of root vectors of $U_h({\frak g})$  in terms of  a braid group action on $U_h({\frak g})$. Let $m_{ij}$, $i\neq j$ be equal to $2,3,4,6$ if $a_{ij}a_{ji}$ is equal to $0,1,2,3$, respectively. The braid group $\mathcal{B}_\g$ associated to $\g$ has generators $T_i$, $i=1,\ldots, l$, and defining relations
$$
T_iT_jT_iT_j\ldots=T_jT_iT_jT_i\ldots
$$
for all $i\neq j$, where there are $m_{ij}$ $T$'s on each side of the equation.

$\mathcal{B}_\g$ acts by algebra automorphisms of $U_h({\frak g})$ as follows:
\begin{eqnarray*}
T_i(X_i^+)=-X_i^-e^{hd_iH_i},~T_i(X_i^-)=-e^{-hd_iH_i}X_i^+,~T_i(H_j)=H_j-a_{ji}H_i, \\
\\
T_i(X_j^+)=\sum_{r=0}^{-a_{ij}}(-1)^{r-a_{ij}}q_i^{-r}
(X_i^+ )^{(-a_{ij}-r)}X_j^+(X_i^+)^{(r)},~i\neq j,\\
\\
T_i(X_j^-)=\sum_{r=0}^{-a_{ij}}(-1)^{r-a_{ij}}q_i^{r}
(X_i^-)^{(r)}X_j^-(X_i^-)^{(-a_{ij}-r)},~i\neq j,
\end{eqnarray*}
where
$$
(X_i^+)^{(r)}=\frac{(X_i^+)^{r}}{[r]_{q_i}!},~(X_i^-)^{(r)}=\frac{(X_i^-)^{r}}{[r]_{q_i}!},~r\geq 0,~i=1,\ldots,l.
$$

Recall that an ordering of a set of positive roots $\Delta_+$ is called normal if
for any three roots $\alpha,~\beta,~\gamma$ such that
$\gamma=\alpha+\beta$ we have either $\alpha<\gamma<\beta$ or $\beta<\gamma<\alpha$.

For any reduced decomposition $w_0=s_{i_1}\ldots s_{i_D}$ of the longest element $w_0$ of the Weyl group $W$ of $\g$ the ordering
$$
\beta_1=\alpha_{i_1},\beta_2=s_{i_1}\alpha_{i_2},\ldots,\beta_D=s_{i_1}\ldots s_{i_{D-1}}\alpha_{i_D}
$$
is a normal ordering in $\Delta_+$, and there is a one--to--one correspondence between normal orderings of $\Delta_+$ and reduced decompositions of $w_0$ (see \cite{Z1}).

Fix a reduced decomposition $w_0=s_{i_1}\ldots s_{i_D}$ of $w_0$ and define the corresponding root vectors in $U_h({\frak g})$ by
\begin{equation}\label{rootvect}
X_{\beta_k}^\pm=T_{i_1}\ldots T_{i_{k-1}}X_{i_k}^\pm.
\end{equation}

Note that one can construct root vectors in the Lie algebra $\g$ in a similar way. Namely, if $X_{\pm \alpha_i}$ are simple root vectors of $\g$ then one can introduce an action of the braid group $\mathcal{B}_\g$ by algebra automorphisms of ${\frak g}$ defined on the standard generators as follows:
\begin{eqnarray*}
T_i(X_{\pm \alpha_i})=-X_{\mp \alpha_i},~T_i(H_j)=H_j-a_{ji}H_i, \\
\\
T_i(X_{\alpha_j})=\frac{1}{(-a_{ij})!}
{\rm ad}_{X_{\alpha_i} }^{-a_{ij}}X_{\alpha_j},~i\neq j,\\
\\
T_i(X_{-\alpha_j})=\frac{(-1)^{a_{ij}}}{(-a_{ij})!}
{\rm ad}_{X_{-\alpha_i} }^{-a_{ij}}X_{-\alpha_j},~i\neq j.
\end{eqnarray*}
Now the root vectors $X_{\pm \beta_k}\in \g_{\pm \beta_k}$ of $\g$ can be defined by
\begin{equation}\label{rootvectg}
X_{\pm \beta_k}=T_{i_1}\ldots T_{i_{k-1}}X_{\pm \alpha_{i_k}}.
\end{equation}

The root vectors $X_{\beta}^-$ satisfy the following relations:
\begin{equation}\label{qcom}
X_{\alpha}^-X_{\beta}^- - q^{(\alpha,\beta)}X_{\beta}^-X_{\alpha}^-= \sum_{\alpha<\delta_1<\ldots<\delta_n<\beta}C(k_1,\ldots,k_n)
{(X_{\delta_1}^-)}^{k_1}{(X_{\delta_2}^-)}^{k_2}\ldots {(X_{\delta_n}^-)}^{k_n},~\alpha<\beta,
\end{equation}

where $C(k_1,\ldots,k_n)\in {\mathcal{P}}$, and ${\mathcal{P}}={\Bbb C}[q,q^{-1}]$ if $\g$ is simply-laced, ${\mathcal{P}}={\Bbb C}[q,q^{-1}, \frac{1}{[2]_q}]$ if $\g$ is of type $B_l,C_l$ or $F_4$, and ${\mathcal{P}}={\Bbb C}[q,q^{-1}, \frac{1}{[2]_q}, \frac{1}{[3]_q}]$ if $\g$ is of type $G_2$.

Note that by construction
$$
\begin{array}{l}
X_\beta^+~(\mbox{mod }h)=X_\beta \in {\frak g}_\beta,\\
\\
X_\beta^-~(\mbox{mod }h)=X_{-\beta} \in {\frak g}_{-\beta}
\end{array}
$$
are root vectors of $\frak g$.

$U_h({\frak g})$ is a quasitriangular Hopf algebra, i.e. there exists an invertible element
${\mathcal R}\in U_h({\frak g})\otimes U_h({\frak g})$, called a universal R--matrix, such that
\begin{equation}\label{quasitr}
\Delta^{opp}_h(a)={\mathcal R}\Delta_h(a){\mathcal R}^{-1}\mbox{ for all } a\in U_h({\frak g}).
\end{equation}

An explicit expression for $\mathcal R$ may be written by making use of the q--exponential
$$
exp_q(x)=\sum_{k=0}^\infty q^{\frac{1}{2}k(k+1)}{x^k \over [k]_q!}
$$
in terms of which the element $\mathcal R$ takes the form (see e.g. \cite{ChP}, Theorem 8.3.9):
\begin{equation}\label{univr}
{\mathcal R}=\prod_{\beta}
exp_{q_{\beta}}[(1-q_{\beta}^{-2})X_{\beta}^-\otimes X_{\beta}^+]exp\left[h\sum_{i=1}^l(Y_i\otimes H_i)\right],
\end{equation}
where
the product is over all the positive roots of $\frak g$, and the order of the terms is such that
the $\alpha$--term appears to the left of the $\beta$--term if $\alpha < \beta$ with respect to the normal
ordering of $\Delta_+$.

One can calculate the action of the comultiplication on the root vectors $X_{\beta_k}^\pm$ in terms of the universal R--matrix. For instance for $\Delta_h(X_{\beta_k}^-)$ one has (see e.g. \cite{ChP}, Theorem 8.3.7)
\begin{equation}\label{comult}
\Delta_h(X_{\beta_k}^-)=\widetilde{\mathcal R}_{<\beta_k}(X_{\beta_k}^-\otimes 1+e^{h\beta^\vee}\otimes X_{\beta_k}^-)\widetilde{\mathcal R}^{-1}_{<\beta_k},
\end{equation}
where
$$
\widetilde{\mathcal R}_{<\beta_k}=\widetilde{\mathcal R}_{\beta_{1}}\ldots \widetilde{\mathcal R}_{\beta_{k-1}},~\widetilde{\mathcal R}_{\beta_r}=exp_{q_{\beta_r}}[(1-q_{\beta_r}^{-2})X_{\beta_r}^+\otimes X_{\beta_r}^-].
$$


\section{Realizations of quantum groups associated to Weyl group elements}\label{wqreal}

\setcounter{equation}{0}
\setcounter{theorem}{0}

Some important ingredients that will be used in the proof of the main statement in Section \ref{WHITT} are certain subalgebras of the quantum group. These subalgebras are defined in terms of realizations of the algebra $U_h(\g)$ associated to Weyl group elements. Following \cite{S9} we introduce these realizations in this section.

Let $s$ be an element of the Weyl group $W$ of the pair $(\g,\h)$, and $\h'$ the orthogonal complement in $\h$, with respect to the Killing form, to the subspace of $\h$ fixed by the natural action of $s$ on $\h$. Let $\h'^*$ be the image of $\h'$ in $\h^*$ under the identification $\h^*\simeq \h$ induced by the canonical bilinear form on $\g$.
The restriction of the natural action of $s$ on $\h^*$ to the subspace $\h'^*$ has no fixed points. Therefore one can define the Cayley transform ${1+s \over 1-s }P_{{\h'}^*}:\h^*\rightarrow {{\h'}^*}\subset \h^*$ of the restriction of $s$ to ${\h'}^*$, where $P_{{\h'}^*}$ is the orthogonal projection operator onto ${{\h'}^*}$ in $\h^*$, with respect to the Killing form.

Recall that in the
classification theory of conjugacy classes in the Weyl group $W$ of the complex simple Lie algebra $\g$
the so-called primitive (or semi--Coxeter in another terminology) elements play a primary role.
According to the results of \cite{C} the element $s$ of the Weyl group of the pair $(\g,\h)$ is primitive in the
Weyl group $W'$ of a regular semisimple Lie subalgebra $\g'\subset \g$, ${\rm rank}~\g'={\rm dim}~\h'$, of the form
$$
\g'=\h'+\sum_{\alpha\in \Delta'}\g_\alpha,
$$
where $\Delta'$ is a root subsystem of the root system $\Delta$ of $\g$, $\g_\alpha$ is the root subspace of $\g$ corresponding to root $\alpha$.

Moreover, by Theorem C in \cite{C} $s$ can be represented as a product of two involutions,
\begin{equation}\label{inv}
s=s^1s^2,
\end{equation}
where $s^1=s_{\gamma_1}\ldots s_{\gamma_n}$, $s^2=s_{\gamma_{n+1}}\ldots s_{\gamma_{l'}}$, the roots in each of the sets $\gamma_1, \ldots \gamma_n$ and ${\gamma_{n+1}}\ldots {\gamma_{l'}}$ are positive and mutually orthogonal, and
the roots $\gamma_1, \ldots \gamma_{l'}$ form a linear basis of ${\h'}^*$, in particular $l'$ is the rank of $\g'$.
The scalar products $\left( {1+s \over 1-s }P_{{\h'}^*}\gamma_i , \gamma_j \right)$ can be computed as follows.
\begin{lemma}{\bf (\cite{S9}, Lemma 6.2)}\label{tmatrel}
Let $P_{{\h'}^*}$ be the orthogonal projection operator onto ${{\h'}^*}$ in $\h^*$, with respect to the Killing form.
Then the scalar products $\left( {1+s \over 1-s }P_{{\h'}^*}\gamma_i , \gamma_j \right)$ are of the form:
\begin{equation}\label{matrel}
\left( {1+s \over 1-s }P_{{\h'}^*}\gamma_i , \gamma_j \right)=
\varepsilon_{ij}(\gamma_i,\gamma_j),
\end{equation}
where
$$
\varepsilon_{ij} =\left\{ \begin{array}{ll}
-1 & i <j \\
0 & i=j \\
1 & i >j
\end{array}
\right  . .
$$
\end{lemma}

Let $\gamma_i^*$, $i=1,\ldots, l'$ be the basis of $\h'^*$ dual to $\gamma_i$, $i=1,\ldots, l'$ with respect to the restriction of the bilinear form $(\cdot,\cdot)$ to $\h'^*$. Since the numbers $(\gamma_i,\gamma_j)$ are integer each element $\gamma_i^*$ has the form $\gamma_i^*=\sum_{j=1}^{l'}m_{ij}\gamma_j$, where $m_{ij}\in \mathbb{Q}$. Therefore by Lemma \ref{tmatrel} and using for simple roots $\alpha_i$ the decomposition of the form $P_{{\h'}^*}\alpha_i=\sum_{p=1}^{l'}(\alpha_i,\gamma_p)\gamma_p^*=\sum_{p,q=1}^{l'}(\alpha_i,\gamma_p)m_{pq}\gamma_q$ we deduce that the numbers
\begin{eqnarray}
q_{ij}=\frac{1}{d_j}\left( {1+s \over 1-s }P_{{\h'}^*}\alpha_i,\alpha_j\right)= \qquad \qquad \qquad \qquad \qquad \qquad \qquad \label{pij} \\ \qquad \qquad =\frac{1}{d_j}\sum_{k,l,p,q=1}^{l'}(\gamma_k,\alpha_i)(\gamma_l,\alpha_j)\left( {1+s \over 1-s }P_{{\h'}^*}\gamma_p,\gamma_q\right)m_{kp}m_{lq},~i,j=1,\ldots,l \nonumber
\end{eqnarray}
are rational as all factors in the products in the sum in the right hand side are rational.  
Fix a positive integer $d$ such that $q_{ij}\in \frac{1}{d}\mathbb{Z}$ for any $i<j$ or for any $i>j$, $i,j=1,\ldots ,l$.

Let $n$ be a non-zero integer number, and
let
$U_h^{s}({\frak g})$ be the topological algebra over ${\Bbb C}[[h]]$ topologically generated by elements
$e_i , f_i , H_i,~i=1, \ldots l$ subject to the relations:
$$
\begin{array}{l}
[H_i,H_j]=0,~~ [H_i,e_j]=a_{ij}e_j, ~~ [H_i,f_j]=-a_{ij}f_j,~~e_i f_j -q^{ c_{ij}} f_j e_i = \delta _{i,j}{K_i -K_i^{-1} \over q_i -q_i^{-1}} ,\\
\\
c_{ij}=nd\left( {1+s \over 1-s }P_{{\h'}^*}\alpha_i , \alpha_j \right),~~K_i=e^{d_ihH_i}, \\
\\
\sum_{r=0}^{1-a_{ij}}(-1)^r q^{r c_{ij}}
\left[ \begin{array}{c} 1-a_{ij} \\ r \end{array} \right]_{q_i}
(e_i )^{1-a_{ij}-r}e_j (e_i)^r =0 ,~ i \neq j , \\
\\
\sum_{r=0}^{1-a_{ij}}(-1)^r q^{r c_{ij}}
\left[ \begin{array}{c} 1-a_{ij} \\ r \end{array} \right]_{q_i}
(f_i )^{1-a_{ij}-r}f_j (f_i)^r =0 ,~ i \neq j .
\end{array}
$$
Note that the matrix $c_{ij}$ is skew--symmetric.

\begin{proposition} {\bf (\cite{S9}, Theorem 4.1)} \label{newreal}
For every solution $n_{ij}\in {\Bbb C},~i,j=1,\ldots ,l$ of equations
\begin{equation}\label{eqpi}
d_jn_{ij}-d_in_{ji}=c_{ij}
\end{equation}
there exists an algebra
isomorphism $\psi_{\{ n_{ij}\}} : U_h^{s}({\frak g}) \rightarrow
U_h({\frak g})$ defined  by the formulas:
$$
\psi_{\{ n_{ij}\}}(e_i)=X_i^+ \prod_{p=1}^le^{hn_{ip}Y_p},~~
\psi_{\{ n_{ij}\}}(f_i)=\prod_{p=1}^le^{-hn_{ip}Y_p}X_i^- ,~~
\psi_{\{ n_{ij}\}}(H_i)=H_i, i=1,\ldots ,l .
$$
\end{proposition}

The algebra $U_h^{s}({\frak g})$ is called the realization of the quantum group $U_h({\frak g})$ corresponding to the element $s\in W$.

\begin{remark}\label{auts}
Let $n_{ij}'\in {\Bbb C}$ be a solution of the homogeneous system that corresponds to (\ref{eqpi}),
\begin{equation}\label{homeq}
d_in_{ji}'-d_jn_{ij}'=0.
\end{equation}
Then the map defined by
\begin{equation}\label{sautom}
\begin{array}{l}
X_i^+ \mapsto X_i^+ \prod_{p=1}^le^{hn_{ip}'Y_p},\\
 \\
X_i^- \mapsto \prod_{p=1}^le^{-hn_{ip}'Y_p}X_i^- , \\
 \\
H_i \mapsto H_i
\end{array}
\end{equation}
is an automorphism of $U_h({\frak g})$. Therefore for given element $s\in W$ the isomorphism $\psi_{\{ n_{ij}\}}$
is defined uniquely up to automorphisms (\ref{sautom}) of $U_h({\frak g})$.
\end{remark}

The realizations $U_h^{s}({\frak g})$ of the quantum group $U_h({\frak g})$
are connected with quantizations of some nonstandard bialgebra structures on $\frak g$. At the quantum level
changing bialgebra structure corresponds to the so--called Drinfeld twist (see \cite{S9}, Section 4).

Equip $U_h^{s}({\frak g})$ with the comultiplication $\Delta_{s}$ given by
$$
\begin{array}{c}
\Delta_{s}(H_i)=H_i\otimes 1+1\otimes H_i,\\
\\
\Delta_{s}(e_i)=e_i\otimes e^{-hd_iH_i}+e^{hd_i nd{1+s \over 1-s}P_{{\h'}}H_i}\otimes e_i,~~
\Delta_{s}(f_i)=f_i\otimes 1+e^{-hd_i nd{1+s \over 1-s}P_{{\h'}}H_i+hd_iH_i}\otimes f_i,
\end{array}
$$
the antipode $S_s(x)$ given by
$$
S_s(e_i)=-e^{-hd_i nd{1+s \over 1-s}P_{{\h'}}H_i}e_ie^{hd_iH_i},~S_s(f_i)=-e^{hd_i nd{1+s \over 1-s}P_{{\h'}}H_i-hd_iH_i}f_i,~S_s(H_i)=-H_i,
$$
and counit defined by
$$
\varepsilon_{s}(H_i)=\varepsilon_s(X_i^\pm)=0.
$$

The comultiplication $\Delta_{s}$ is obtained from the standard comultiplication by a Drinfeld twist (see \cite{S9}, Section 4). Namely, let
\begin{equation}\label{Ftw}
{\mathcal F}=exp(-h\sum_{i,j=1}^l {n_{ij} \over d_i}Y_i\otimes Y_j) \in U_h({\frak h})\otimes U_h({\frak h}),
\end{equation}
where $n_{ij}$ is the solution of the corresponding equation (\ref{eqpi}) used in the definition of the isomorphism $\psi_{\{ n_{ij}\}}$. Then
\begin{equation}\label{defds}
\Delta_s(a)=(\psi_{\{ n_{ij}\}}^{-1}\otimes \psi_{\{ n_{ij}\}}^{-1}){\mathcal F}\Delta(\psi_{\{ n_{ij}\}}(a)){\mathcal F}^{-1}.
\end{equation}

We shall construct analogues of root vectors for $U_h^{s}({\frak g})$.
It is convenient to introduce an operator $K\in {\rm End}~{\frak h}$ defined by
\begin{equation}\label{Kdef}
KH_i=\sum_{j=1}^l{n_{ij} \over d_i}Y_j.
\end{equation}

For any solution of equation (\ref{eqpi}) and any normal ordering of the root system $\Delta_+$
we define the following elements of $U_h^{s}({\frak g})$, $e_{\beta}=\psi_{\{ n\}}^{-1}(X_{\beta}^+e^{hK\beta^\vee})$, $f_{\beta}=\psi_{\{ n\}}^{-1}(e^{-hK\beta^\vee}X_{\beta}^-),~\beta \in \Delta_+$.

$U_h^{s}({\frak g})$ is a quasitriangular topological Hopf algebra with the universal R--matrix
${\mathcal R}^{s}$,
\begin{equation}\label{rmatrspi}
\begin{array}{l}
{\mathcal R}^{s}=
\prod_{\beta}
exp_{q_{\beta}}[(1-q_{\beta}^{-2})f_{\beta} \otimes
e_{\beta}e^{-hnd{1+s \over 1-s}P_{{\h'}} \beta^\vee}]\times \\
exp\left[ h(\sum_{i=1}^l(Y_i\otimes H_i)-
\sum_{i=1}^l nd{1+s \over 1-s }P_{{\h'}}H_i\otimes Y_i) \right],
\end{array}
\end{equation}
where $P_{{\h'}}$ is the orthogonal projection operator onto $\h'$ in $\h$ with respect to the Killing form, and the order of the terms is such that
the $\alpha$--term appears to the left of the $\beta$--term if $\alpha < \beta$ with respect to the normal
ordering of $\Delta_+$.

Using formula (\ref{comult}) and Proposition 4.3 in \cite{S9} one can also find that
\begin{equation}\label{comults}
\Delta_s(f_{\beta_k})=\widetilde{\mathcal R}_{<\beta_k}^s(e^{-hnd{1+s \over 1-s}P_{{\h'}}\beta^\vee+h\beta^\vee} \otimes f_{\beta_k} + f_{\beta_k}\otimes 1)(\widetilde{\mathcal R}^s_{<\beta_k})^{-1},
\end{equation}
where
$$
\widetilde{\mathcal R}^s_{<\beta_k}=\widetilde{\mathcal R}^s_{\beta_{1}}\ldots \widetilde{\mathcal R}^s_{\beta_{k-1}},~\widetilde{\mathcal R}^s_{\beta_r}=exp_{q_{\beta_r}}[(1-q_{\beta_r}^{-2})e_{\beta_r}e^{-hnd{1+s \over 1-s}P_{{\h'}}\beta^\vee}\otimes f_{\beta_r}],
$$
and
$$
(\widetilde{\mathcal R}_{<\beta_k}^s)^{-1}=(\widetilde{\mathcal R}_{\beta_{k-1}}^s)^{-1}\ldots (\widetilde{\mathcal R}^s_{\beta_{1}})^{-1},~(\widetilde{\mathcal R}^s_{\beta_r})^{-1}=exp_{q_{\beta_r}^{-1}}[(1-q_{\beta_r}^{2})e_{\beta_r}e^{-hnd{1+s \over 1-s}P_{{\h'}}\beta^\vee}\otimes f_{\beta_r}].
$$

We shall actually need not the algebras $U_h({\frak g})$ and $U_h^{s}({\frak g})$ themselves but some their specializations defined over certain rings and over the field of complex numbers. They are similar to the  non-restricted integral form and to its specialization for the standard quantum group $U_h(\g)$. The results below are slight modifications of similar statements for $U_h(\g)$, and we refer to \cite{ChP}, Ch. 9 for the proofs.

Note that by the choice of $d$ we have $q^{c_{ij}}\in \mathbb{C}[q,q^{-1}]$.
Let ${\mathcal{A}}={\Bbb C}[q,q^{-1}]$ if $\g$ is simply-laced, ${\mathcal{A}}={\Bbb C}[q,q^{-1}, \frac{1}{[2]_q}]$ if $\g$ is of type $B_l,C_l$ or $F_4$, and ${\mathcal{A}}={\Bbb C}[q,q^{-1}, \frac{1}{[2]_q}, \frac{1}{[3]_q}]$ if $\g$ is of type $G_2$.

Let $U_{{\mathcal{A}}}^{s}({\frak g})$ be the ${\mathcal{A}}$-subalgebra of $U_h^{s}({\frak g})$ generated by the elements $e_i , f_i , L_i^{\pm 1}, ~{K_i -K_i^{-1} \over q_i -q_i^{-1}},~i=1, \ldots, l$ and
$U_{{\mathcal{A}}}({\frak g})$ the ${\mathcal{A}}$-subalgebra of $U_h({\frak g})$ generated by the elements $X_i^\pm , L_i^{\pm 1}, ~{K_i -K_i^{-1} \over q_i -q_i^{-1}},~i=1, \ldots, l$.

\begin{remark}\label{nsol}
Note that the general solution of equation (\ref{eqpi})
is given by
\begin{equation}\label{eq3}
n_{ij}=\frac 1{2d_j} (c_{ij} + {s_{ij}}),
\end{equation}
where $s_{ij}=s_{ji}$.
If $q_{ij}\in \frac{1}{d}\mathbb{Z}$ for any $i<j$, we put
$$
s_{ij} =\left\{ \begin{array}{ll}
c_{ij} & i <j \\
0 & i=j \\
-c_{ij} & i >j
\end{array}
\right  . .
$$
Then
$$
n_{ij} =\left\{ \begin{array}{ll}
\frac 1{d_j}c_{ij} & i <j \\
0 & i=j \\
0 & i >j
\end{array}
\right  . .
$$
By the choice of $d$ we have $c_{ij}\in d_jn\mathbb{Z}$, and hence $n_{ij}\in n\mathbb{Z}$ for $i,j=1,\ldots ,l$.
A similar consideration shows that if $q_{ij}\in \frac{1}{d}\mathbb{Z}$ for any $i>j$ then there exists a solution $n_{ij}\in n\mathbb{Z}$ for $i,j=1,\ldots ,l$.
\end{remark}

Since there is a solution $n_{ij}\in n\mathbb{Z}$ for $i,j=1,\ldots ,l$ to equation (\ref{eqpi}) the map $\psi_{\{n_{ij}\}}$ associated to this solution in Proposition \ref{newreal} induces an isomorphism of algebras $\psi_{\{n_{ij}\}}:U_{{\mathcal{A}}}^s({\frak g})\rightarrow U_{{\mathcal{A}}}({\frak g})$.

All algebras introduced above are Hopf algebras with comultiplications induced from $U_h^s(\g)$ or $U_h({\frak g})$. The algebra $U_{{\mathcal{A}}}^s({\frak g})$ acts on itself by the left and the right adjoint representations,
\begin{equation}\label{Adact}
{\rm Ad}'x(y)=x_1yS_s(x_2),~{\rm Ad}x(y)=S_s(x_1)yx_2,
\end{equation}
where we use the Swedler notation for the comultiplication, $\Delta_s(x)=x_1\otimes x_2$.

Denote by $U_{\mathcal{A}}({\frak n}_+), U_{\mathcal{A}}({\frak n}_-)$ the subalgebras of $U_{\mathcal{A}}({\frak g})$ generated by the
$X_i^+$ and by the $X_i^-$, respectively. For any $\alpha\in \Delta_+$ and the quantum root vectors $X_{\alpha}^\pm$ constructed with the help of any normal ordering in $\Delta_+$ one has $X_{\alpha}^\pm\in U_{\mathcal{A}}({\frak n}_\pm)$.

Let $U_{\mathcal{A}}^{s}({\frak n}_+),U_{\mathcal{A}}^{s}({\frak n}_-)$ be the subalgebras of $U_{\mathcal{A}}^{s}({\frak g})$ generated by the
$e_i$ and by the $f_i$, $i=1,\ldots,l$, respectively.

From the definition of the quantum root vectors $X_\beta^\pm$ and from the commutation relations between the generators $X_i^\pm$ and $L_j^{\pm 1}$, $i,j=1,\ldots l$ it follows that for the solution of equation (\ref{eqpi}) introduced in Remark \ref{nsol} and any normal ordering of the root system $\Delta_+$
the elements $e_{\beta}=\psi_{\{ n_{ij}\}}^{-1}(X_{\beta}^+e^{hK\beta^\vee})$ and
$f_{\beta}=\psi_{\{ n_{ij}\}}^{-1}(e^{-hK\beta^\vee}X_{\beta}^-),~\beta \in \Delta_+$
lie in the subalgebras $U_{\mathcal{A}}^{s}({\frak n}_+)$ and $U_{\mathcal{A}}^{s}({\frak n}_-)$, respectively.

The most important for us are the specializations $U_\varepsilon(\g)$ and  $U_\varepsilon^s(\g)$ of $U_{\mathcal{A}}(\g)$ and $U_{\mathcal{A}}^{s}(\g)$, $U_\varepsilon(\g)=U_{\mathcal{A}}(\g)/(q-\varepsilon)U_{\mathcal{A}}(\g)$  $U_\varepsilon^s(\g)=U_{\mathcal{A}}^{s}(\g)/(q-\varepsilon)U_{\mathcal{A}}^{s}(\g)$, where $\varepsilon\in \mathbb{C}^*$, $\varepsilon^{2d_i}\neq 1$, $i=1,\ldots, l$ and $\varepsilon^{4}\neq 1$ if $\g$ is of type $G_2$.

Note that all specializations introduced above are invariant under the action of the braid group $\mathcal{B}_\g$.

Denote by $U_{\varepsilon}({\frak n}_+),U_{\varepsilon}({\frak n}_-)$ and $U_{\varepsilon}(\h)$ the subalgebras of $U_{\varepsilon}({\frak g})$ generated by the
$X_i^+$, by the $X_i^-$, and by the $L_i^{\pm 1}$, respectively.

Fix a normal ordering in $\Delta_+$ and let $X_\alpha^\pm$ be the corresponding quantum root vectors
Define for ${\bf r}=(r_1,\ldots ,r_D)\in {\Bbb N}^D$,
$$
(X^+)^{\bf r}=(X_{\beta_1}^+)^{r_1}\ldots (X_{\beta_D}^+)^{r_D},
$$
$$
(X^-)^{\bf r}=(X_{\beta_D}^-)^{r_D}\ldots (X_{\beta_1}^-)^{r_1},
$$
and for ${\bf s}=(s_1,\ldots s_l)\in {\Bbb Z}^{l}$,
$$
L^{{\bf s}}=L_1^{{s_1}}\ldots L_l^{{s_l}}.
$$

\begin{proposition}{\bf (\cite{ChP}, Proposition 9.2.2)}\label{PBW1}
The algebra $U_\varepsilon(\g)$ is generated over $\mathbb{C}$ by $L_i^{\pm 1},~X_i^\pm,~i=1,\ldots ,l$.

The elements $(X^+)^{\bf r}$, $(X^-)^{\bf t}$ and $L^{{\bf s}}$, for ${\bf r},~{\bf t}\in {\Bbb N}^D$,
${\bf s}\in {\Bbb Z}^l$, form linear bases of $U_\varepsilon({\frak n}_+),U_\varepsilon({\frak n}_-)$ and $U_\varepsilon({\frak h})$,
respectively, and the products $(X^+)^{\bf r}L^{{\bf s}}(X^-)^{\bf t}$ form a basis of
$U_\varepsilon({\frak g})$. In particular, multiplication defines an isomorphism of vector spaces:
$$
U_\varepsilon({\frak n}_-)\otimes U_\varepsilon({\frak h}) \otimes U_\varepsilon({\frak n}_+)\rightarrow U_\varepsilon({\frak g}).
$$
The root vectors $X_{\beta}^-$ satisfy the following relations in $U_\varepsilon({\frak g})$:
\begin{equation}\label{qcom1}
X_{\alpha}^-X_{\beta}^- - \varepsilon^{(\alpha,\beta)}X_{\beta}^-X_{\alpha}^-= \sum_{\alpha<\delta_1<\ldots<\delta_n<\beta}C(k_1,\ldots,k_n)
{(X_{\delta_n}^-)}^{k_n}{(X_{\delta_{n-1}}^-)}^{k_{n-1}}\ldots {(X_{\delta_1}^-)}^{k_1},
\end{equation}
where $\alpha<\beta$, the sum is taken over tuples of roots $\delta_1,\ldots, \delta_n$ such that $\alpha<\delta_1<\ldots<\delta_n<\beta$, and over $k_i\in \Bbb{N}$, $C(k_1,\ldots,k_n)\in {\Bbb C}$, and for each term in the right hand side $\sum_{i=1}^nk_i\delta_i=\alpha+\beta$.
\end{proposition}

Now we shall study the algebraic structure of $U_\varepsilon^{s}({\frak g})$.
Denote by $U_\varepsilon^{s}({\frak n}_\pm) $ the subalgebra in $U_\varepsilon^{s}({\frak g})$ generated by
$e_i ~(f_i) ,i=1, \ldots l$.
Let $U_\varepsilon^{s}({\frak h})$ be the subalgebra in $U_\varepsilon^{s}({\frak g})$ generated by $L_i^{\pm 1},~i=1,\ldots ,l$.

We shall construct a Poincar\'{e}--Birkhoff-Witt basis for $U_\varepsilon^{s}({\frak g})$.

\begin{proposition}\label{rootss}
(i) For any normal ordering of the root system $\Delta_+$ and for any integer valued solution of equation (\ref{eqpi})
the elements $e_{\beta}$ and
$f_{\beta}$
lie in the subalgebras $U_\varepsilon^{s}({\frak n}_+)$ and $U_\varepsilon^{s}({\frak n}_-)$, respectively.
The elements $f_{\beta},~\beta \in \Delta_+$ satisfy the following commutation relations
\begin{equation}\label{erel}
f_{\alpha}f_{\beta} - \varepsilon^{(\alpha,\beta)+nd({1+s \over 1-s}P_{{\h'}^*}\alpha,\beta)}f_{\beta}f_{\alpha}= \sum_{\alpha<\delta_1<\ldots<\delta_n<\beta}C'(k_1,\ldots,k_n)
f_{\delta_n}^{k_n}f_{\delta_{n-1}}^{k_{n-1}}\ldots f_{\delta_1}^{k_1},~\alpha <\beta,
\end{equation}
where $C'(k_1,\ldots,k_n)\in \mathbb{C}$.

(ii) Moreover, the elements
$(e)^{\bf r}=(e_{\beta_1})^{r_1}\ldots (e_{\beta_D})^{r_D},~~(f)^{\bf t}=(f_{\beta_D})^{t_D}\ldots (f_{\beta_1})^{t_1}$
and $L^{\bf s}=L_1^{s_1}\ldots L_l^{s_l}$
for ${\bf r},~{\bf t}\in {\Bbb N}^D$, ${\bf s}\in {\Bbb Z}^l$ form bases of $U_\varepsilon^{s}({\frak n}_+),~U_\varepsilon^{s}({\frak n}_-)$ and $U_\varepsilon^{s}({\frak h})$,
and the products $(f)^{\bf t}L^{\bf s}(e)^{\bf r}$ form a basis of
$U_\varepsilon^{s}({\frak g})$. In particular, multiplication defines an isomorphism of vector spaces, $$
U_\varepsilon^{s}({\frak n}_-)\otimes U_\varepsilon^{s}({\frak h})\otimes U_\varepsilon^{s}({\frak n}_+)\rightarrow U_\varepsilon^{s}({\frak g}).
$$
\end{proposition}

The proof of this proposition is similar to the proof of Proposition 4.2 in \cite{S9}.


\section{Nilpotent subalgebras and quantum groups}\label{nilpuq}

\setcounter{equation}{0}
\setcounter{theorem}{0}

In this section we define the subalgebras of $U_\varepsilon({\frak g})$ which resemble nilpotent subalgebras in
${\frak g}$ and possess non--trivial characters. We start by recalling the definition of certain normal orderings of root systems associated to Weyl group elements (see \cite{S9}, Section 5 for more details). The definition of subalgebras of $U_\varepsilon(\g)$ having non--trivial characters will be given in terms of root vectors associated to such normal orderings.

\begin{proposition}\label{pord}{\bf (\cite{S9}, Proposition 5.1)}
Let $s\in W$ be an element of the Weyl group $W$ of the pair $(\g,\h)$, $\Delta$ the root system of the pair $(\g,\h)$ Then there is a system of positive roots $\Delta_+^s$ such that decomposition (\ref{inv}) is reduced in the sense that ${l}(s)={l}(s^2)+{l}(s^1)$, where ${l}(\cdot)$ is the length function in $W$ with respect to the system of simple roots in $\Delta_+^s$, and $\Delta_{s}^s=\Delta_{s^{2}}^s\bigcup s^2(\Delta_{s^{1}}^s)$, $\Delta_{s^{-1}}^s=\Delta_{s^{1}}^s\bigcup s^1(\Delta_{s^{2}}^s)$ (disjoint unions), $\Delta_{s^{1,2}}^s=\{\alpha \in \Delta_+^s:s^{1,2}\alpha \in -\Delta_+^s\}$, $\Delta_{s}^s=\{\alpha \in \Delta_+^s:s\alpha \in -\Delta_+^s\}$. Here $s^1,s^2$ are the involutions entering decomposition (\ref{inv}), $s^1=s_{\gamma_1}\ldots s_{\gamma_n}$, $s^2=s_{\gamma_{n+1}}\ldots s_{\gamma_{l'}}$, the roots in each of the sets $\gamma_1, \ldots, \gamma_n$ and ${\gamma_{n+1}},\ldots, {\gamma_{l'}}$ are positive and mutually orthogonal.

Moreover, there is a normal ordering of the root system $\Delta_+^s$ of the following form
\begin{eqnarray}
\beta_1^1,\ldots, \beta_t^1,\beta_{t+1}^1, \ldots,\beta_{t+\frac{p-n}{2}}^1, \gamma_1,\beta_{t+\frac{p-n}{2}+2}^1, \ldots , \beta_{t+\frac{p-n}{2}+n_1}^1, \gamma_2, \nonumber \\
\beta_{t+\frac{p-n}{2}+n_1+2}^1 \ldots , \beta_{t+\frac{p-n}{2}+n_2}^1, \gamma_3,\ldots, \gamma_n, \beta_{t+p+1}^1,\ldots, \beta_{l(s^1)}^1,\ldots, \label{NO} \\
\beta_1^2,\ldots, \beta_q^2, \gamma_{n+1},\beta_{q+2}^2, \ldots , \beta_{q+m_1}^2, \gamma_{n+2}, \beta_{q+m_1+2}^2,\ldots , \beta_{q+m_2}^2, \gamma_{n+3},\ldots,  \nonumber \\
\gamma_{l'},\beta_{q+m_{l(s^2)}+1}^2, \ldots,\beta_{2q+2m_{l(s^2)}-(l'-n)}^2, \beta_{2q+2m_{l(s^2)}-(l'-n)+1}^2,\ldots, \beta_{l(s^2)}^2, \nonumber \\
\beta_1^0, \ldots, \beta_{D_0}^0, \nonumber
\end{eqnarray}
where
\begin{eqnarray*}
\{\beta_1^1,\ldots, \beta_t^1,\beta_{t+1}^1, \ldots,\beta_{t+\frac{p-n}{2}}^1, \gamma_1,\beta_{t+\frac{p-n}{2}+2}^1, \ldots , \beta_{t+\frac{p-n}{2}+n_1}^1, \gamma_2, \nonumber \\
\beta_{t+\frac{p-n}{2}+n_1+2}^1 \ldots , \beta_{t+\frac{p-n}{2}+n_2}^1, \gamma_3,\ldots, \gamma_n, \beta_{t+p+1}^1,\ldots, \beta_{l(s^1)}^1\}=\Delta_{s^1}^s,
\end{eqnarray*}
\begin{eqnarray*}
\{\beta_{t+1}^1, \ldots,\beta_{t+\frac{p-n}{2}}^1, \gamma_1,\beta_{t+\frac{p-n}{2}+2}^1, \ldots , \beta_{t+\frac{p-n}{2}+n_1}^1, \gamma_2, \nonumber \\
\beta_{t+\frac{p-n}{2}+n_1+2}^1 \ldots , \beta_{t+\frac{p-n}{2}+n_2}^1, \gamma_3,\ldots, \gamma_n\}=\{\alpha\in \Delta_+^s|s^1(\alpha)=-\alpha\}=\Delta^s_{s^1=-1},
\end{eqnarray*}
\begin{eqnarray*}
\{\beta_1^2,\ldots, \beta_q^2, \gamma_{n+1},\beta_{q+2}^2, \ldots , \beta_{q+m_1}^2, \gamma_{n+2}, \beta_{q+m_1+2}^2,\ldots , \beta_{q+m_2}^2, \gamma_{n+3},\ldots,  \nonumber \\
\gamma_{l'},\beta_{q+m_{l(s^2)}+1}^2, \ldots,\beta_{2q+2m_{l(s^2)}-(l'-n)}^2, \beta_{2q+2m_{l(s^2)}-(l'-n)+1}^2,\ldots, \beta_{l(s^2)}^2\}=\Delta_{s^2}^s,
\end{eqnarray*}
\begin{eqnarray*}
\{\gamma_{n+1},\beta_{q+2}^2, \ldots , \beta_{q+m_1}^2, \gamma_{n+2}, \beta_{q+m_1+2}^2,\ldots , \beta_{q+m_2}^2, \gamma_{n+3},\ldots,  \nonumber \\
\gamma_{l'},\beta_{q+m_{l(s^2)}+1}^2, \ldots,\beta_{2q+2m_{l(s^2)}-(l'-n)}^2\}=\{\alpha\in \Delta_+^s|s^2(\alpha)=-\alpha\}=\Delta^s_{s^2=-1},
\end{eqnarray*}
\begin{equation*}
\{\beta_1^0, \ldots, \beta_{D_0}^0\}=\{\alpha\in \Delta_+^s|s(\alpha)=\alpha\}.
\end{equation*}

The length of the ordered segment $\Delta_{\m_+}\subset \Delta$ in normal ordering (\ref{NO}),
\begin{eqnarray}
\Delta_{\m_+}=\gamma_1,\beta_{t+\frac{p-n}{2}+2}^1, \ldots , \beta_{t+\frac{p-n}{2}+n_1}^1, \gamma_2, \beta_{t+\frac{p-n}{2}+n_1+2}^1 \ldots , \beta_{t+\frac{p-n}{2}+n_2}^1, \nonumber \\
\gamma_3,\ldots, \gamma_n, \beta_{t+p+1}^1,\ldots, \beta_{l(s^1)}^1,\ldots, \beta_1^2,\ldots, \beta_q^2, \label{dn} \\
\gamma_{n+1},\beta_{q+2}^2, \ldots , \beta_{q+m_1}^2, \gamma_{n+2}, \beta_{q+m_1+2}^2,\ldots , \beta_{q+m_2}^2, \gamma_{n+3},\ldots, \gamma_{l'}, \nonumber
\end{eqnarray}
is equal to
\begin{equation}\label{dimm}
D-(\frac{l(s)-l'}{2}+D_0),
\end{equation}
where $D$ is the number of roots in $\Delta_+^s$, $l(s)$ is the length of $s$ and $D_0$ is the number of positive roots fixed by the action of $s$.

For any two roots $\alpha, \beta\in \Delta_{\m_+}$ such that $\alpha<\beta$ the sum $\alpha+\beta$ cannot be represented as a linear combination $\sum_{k=1}^qc_k\gamma_{i_k}$, where $c_k\in \mathbb{N}$ and $\alpha<\gamma_{i_1}<\ldots <\gamma_{i_k}<\beta$.
\end{proposition}

\begin{remark}
In case when $s=s^1$ is an involution the last root in the segment $\Delta_{\m_+}$ is the root preceding $\beta_1^0$ in normal ordering (\ref{NO}).
\end{remark}

We shall also need another system of positive roots associated to (the conjugacy class of) the Weyl group element $s$. In order to define it we need to recall the definition of a circular normal ordering of the root system ${\Delta}$.

Let $\beta_{1}, \beta_{2}, \ldots, \beta_{D}$ be a normal ordering
of a positive root system. Then following \cite{KT3} one can introduce the corresponding circular normal ordering of the root system ${\Delta}$ where
the roots in ${\Delta}$ are located on a circle in
the following way
\begin{center}
\setlength{\unitlength}{0.6mm}
     \begin{picture}(180,120)(-40,0)
     \put(0,50){\makebox(0,0){$\beta_1$}}
     \put(4,69){\makebox(0,0){$\beta_2$}}
     \put(15,85){\circle*{1.5}}
     \put(31,96){\circle*{1.5}}
     \put(50,100){\circle*{1.5}}
     \put(69,96){\circle*{1.5}}
     \put(85,85){\circle*{1.5}}
     \put(96,69){\makebox(0,0){$\beta_D$}}
     \put(100,50){\makebox(0,0){$-\beta_1$}}
     \put(96,31){\makebox(0,0){$-\beta_2$}}
     \put(85,15){\circle*{1.5}}
     \put(69,4){\circle*{1.5}}
     \put(50,0){\circle*{1.5}}
     \put(31,4){\circle*{1.5}}
     \put(15,15){\circle*{1.5}}
     \put(4,31) {\makebox(0,0){-$\beta_D$}}
     \put(64,88){\vector (3,-2){10}}
     \put(36,12){\vector (-3,2){10}}
     \end{picture}
\end{center}
\begin{center}
Circular normal ordering of a root system.
\end{center}

Let $\alpha,\beta\in \Delta$. One says that the segment $[\alpha, \beta]$ of the circle
is minimal if it does not contain the opposite roots $-\alpha$ and $-\beta$ and the root $\beta$ follows after $\alpha$ on the circle above, the circle being oriented clockwise.
In that case one also says that $\alpha < \beta$ in the sense of the circular normal ordering,
\begin{equation}\label{noc}
\alpha < \beta \Leftrightarrow {\rm the ~segment}~ [\alpha, \beta]~{\rm  of~ the ~circle~
is~ minimal}.
\end{equation}

Later we shall need the following property of minimal segments which is a direct consequence of Proposition 3.3 in \cite{kh-t}.
\begin{lemma}\label{minsegm}
Let $[\alpha, \beta]$ be a minimal segment in a circular normal ordering of a root system $\Delta$. Then if $\alpha+\beta$ is a root we have
$$
\alpha<\alpha+\beta<\beta.
$$
\end{lemma}

Note that any segment in a circular normal ordering of $\Delta$ of length equal to the number of positive roots is a system of positive roots.

Now consider the circular normal ordering of $\Delta$ corresponding to the system of positive roots $\Delta_+^s$ and to its normal ordering introduced in Proposition \ref{pord}. The segment which consists of the roots $\alpha$ satisfying $\gamma_1\leq \alpha < -\gamma_1$ is a system of positive roots in $\Delta$ as its length is equal to the number of positive roots and it is closed under addition of roots by Lemma \ref{minsegm}.

The system of positive roots $\Delta_+$ introduced this way and equipped with the normal ordering induced by the circular normal ordering is called the normally ordered system of positive roots associated to the (conjugacy class of) the Weyl group element $s\in W$.

The linear subspace of $\g$ generated by the root vectors $X_{\alpha}$ ($X_{-\alpha}$), $\alpha\in \Delta_{\m_+}$ is in fact a Lie subalgebra ${\m_+}\subset \g$ (${\m_-}\subset \g$).
Let $\b_+$ be the Borel subalgebra associated to $\Delta_+$ and $\b_-$ is the opposite Borel subalgebra. Denote by $\n_\pm$ the nilradicals of $\b_\pm$. Let $H,N_+, N_-,B_+,B_-$ be the Cartan subgroup, the maximal unipotent subgroups and the Borel subgroups of $G$ which correspond to the Lie subalgebras ${\frak h},{{\frak n}_+},{\frak n}_-,{\frak b}_+$ and ${\frak b}_-,$ respectively.
Note that by definition $\Delta_{\m_+} \subset \Delta_+$, and hence ${\m_\pm}\subset \n_\pm$.

Now we can define the subalgebras of $U_\varepsilon({\frak g})$ which resemble nilpotent subalgebras in
${\frak g}$ and possess non--trivial characters.
\begin{theorem}\label{qnil}
Let $s\in W$ be an element of the Weyl group $W$ of the pair $(\g,\h)$, $\Delta$ the root system of the pair $(\g,\h)$. Fix a decomposition (\ref{inv}) of $s$ and let $\Delta_+$ be a system of positive roots associated to $s$. Assume that $\varepsilon^{2d_i}\neq 1$, $\varepsilon^4\neq 1$ if $\g$ is of type $G_2$ and that $\varepsilon^{nd-1}=1$, where $d$ and $n$ are introduced in Section \ref{wqreal}.
Let $U_\varepsilon^{s}({\frak g})$ be the realization of the quantum group $U_\varepsilon({\frak g})$ associated to $s$. Let $f_\beta\in U_\varepsilon^{s}({\n_-})$, $\beta \in \Delta_+$ be the root vectors associated to the corresponding normal ordering of $\Delta_+$.

Then elements $f_\beta\in U_\varepsilon^{s}({\n_-})$, $\beta \in \Delta_{\m_+}$, where $\Delta_{\m_+}\subset \Delta$ is ordered segment (\ref{dn}), generate a subalgebra $U_\varepsilon^{s}({\frak m}_-)\subset U_\varepsilon^{s}({\frak g})$.
The elements
$f^{\bf r}=f_{\beta_D}^{r_D}\ldots f_{\beta_1}^{r_1}$, $r_i\in \mathbb{N}$, $i=1,\ldots D$ and $r_i$ can be strictly positive only if $\beta_i\in \Delta_{\m_+}$, form a linear
basis of $U_\varepsilon^{s}({\frak m}_-)$.

Moreover the map $\chi^s:U_\varepsilon^{s}({\frak m}_-)\rightarrow \mathbb{C}$ defined on generators by
\begin{equation}\label{char}
\chi^s(f_\beta)=\left\{ \begin{array}{ll} 0 & \beta \not \in \{\gamma_1, \ldots, \gamma_{l'}\} \\ c_i & \beta=\gamma_i, c_i\in \mathbb{C}
\end{array}
\right  .
\end{equation}
is a character of $U_\varepsilon^{s}({\frak m}_-)$.
\end{theorem}

\begin{proof}
The first statement of the theorem follows straightforwardly from commutation relations (\ref{erel}) and Proposition \ref{rootss}.

In order to prove that the map $\chi^s:U_\varepsilon^{s}({\frak m}_-)\rightarrow \mathbb{C}$ defined by (\ref{char}) is a character of $U_\varepsilon^{s}({\frak m}_-)$ we show that all relations (\ref{erel}) for $f_\alpha,~f_\beta$ with $\alpha,\beta \in \Delta_{\m_+}$, which are defining relations in the subalgebra $U_\varepsilon^{s}({\frak m}_-)$ by part (ii) of Proposition \ref{rootss},  belong to the kernel of $\chi^s$. By definition the only generators of $U_\varepsilon^{s}({\frak m}_-)$ on which $\chi^s$ may not vanish are $f_{\gamma_i}$, $i=1,\ldots,l'$. By the last statement in Proposition \ref{pord} for any two roots $\alpha, \beta\in \Delta_{\m_+}$ such that $\alpha<\beta$ the sum $\alpha+\beta$ can not be represented as a linear combination $\sum_{k=1}^qc_k\gamma_{i_k}$, where $c_k\in \mathbb{N}$ and $\alpha<\gamma_{i_1}<\ldots <\gamma_{i_k}<\beta$. Hence for any two roots $\alpha, \beta\in \Delta_{\m_+}$ such that $\alpha<\beta$ the value of the map $\chi^s$ on the r.h.s. of the corresponding commutation relation (\ref{erel}) is equal to zero.

Therefore it suffices to prove that
$$
\chi^s(f_{\gamma_i}f_{\gamma_j} - \varepsilon^{(\gamma_i,\gamma_j)+nd({1+s \over 1-s}P_{{\h'}^*}\gamma_i,\gamma_j)}f_{\gamma_j}f_{\gamma_j})=c_ic_j(1-\varepsilon^{(\gamma_i,\gamma_j)+nd({1+s \over 1-s}P_{{\h'}^*}\gamma_i,\gamma_j)})=0,~i<j.
$$

Since $\varepsilon^{nd-1}=1$ and $({1+s \over 1-s}P_{{\h'}^*}\gamma_i,\gamma_j)$ are integer numbers for any $i,j=1,\ldots ,l'$, the last identity always holds provided $(\gamma_i,\gamma_j)+({1+s \over 1-s}P_{{\h'}}^*\gamma_i, \gamma_j)=0$ for $i<j$. As we saw in Lemma \ref{tmatrel} this is indeed the case. This completes the proof.
\end{proof}


\section{Quantum groups at roots of unity}\label{1root}

\setcounter{equation}{0}
\setcounter{theorem}{0}

Let $m$ be an odd positive integer number, and $m>d_i$ is coprime to all $d_i$ for all $i$, $\varepsilon$ a primitive $m$-th root of unity.
In this section, following \cite{ChP}, Section 9.2, we recall some results on the structure of the algebra $U_\varepsilon({\frak g})$. We keep the notation introduced in Section \ref{notation}.

Let $Z_\varepsilon$ be the center of $U_\varepsilon({\frak g})$.

\begin{proposition}{\bf (\cite{DK}, Corollary 3.3, \cite{DKP1}, Theorems 3.5, 7.6 and Proposition 4.5)}\label{ue}
Fix the normal ordering in the positive root system $\Delta_+$ corresponding a reduced decomposition $w_0=s_{i_1}\ldots s_{i_D}$ of the longest element $w_0$ of the Weyl group $W$ of $\g$ and let $X_{\alpha}^\pm$ be the corresponding root vectors in $U_\varepsilon({\frak g})$, and $X_{\alpha}$ the corresponding root vectors in $\g$. Let $x_{\alpha}^-=(\varepsilon_{\alpha}-\varepsilon_{\alpha}^{-1})^m(X_{\alpha}^-)^m$, $x_{\alpha}^+=(\varepsilon_{\alpha}-\varepsilon_{\alpha}^{-1})^mT_0(X_{\alpha}^-)^m$, where $T_0=T_{i_1}\ldots T_{i_D}$, $\alpha\in \Delta_+$ and $l_i=L_i^m$, $i=1,\ldots ,l$ be elements of $U_\varepsilon({\frak g})$.

Then the following statements are true.

(i) The elements $x_{\alpha}^\pm$, $\alpha\in \Delta_+$, $l_i$, $i=1,\ldots ,l$ lie in $Z_\varepsilon$.

(ii) Let $Z_0$  ($Z_0^\pm$ and $Z_0^0$) be the subalgebras of $Z_\varepsilon$ generated by the $x_{\alpha}^\pm$ and the $l_i^{\pm 1}$ (respectively by the $x_{\alpha}^\pm$ and by the $l_i^{\pm 1}$). Then $Z_0^\pm\subset U_\varepsilon({\frak n}_\pm)$, $Z_0^0\subset U_\varepsilon({\frak h})$, $Z_0^\pm$ is the polynomial algebra with generators $x_{\alpha}^\pm$,
$Z_0^0$ is the algebra of Laurent polynomials in the $l_i$, $Z_0^\pm=U_\varepsilon({\frak n}_\pm)\bigcap Z_0$, and multiplication defines an isomorphism of algebras
$$
Z_0^-\otimes Z_0^0 \otimes Z_0^+ \rightarrow Z_0.
$$
The subalgebra $Z_0$ is independent of the choice of the reduced decomposition $w_0=s_{i_1}\ldots s_{i_D}$.

(iii) $U_\varepsilon({\frak g})$ is a free $Z_0$--module with basis the set of monomials $(X^+)^{\bf r}L^{\bf s}(X^-)^{\bf t}$ in the statement of Proposition \ref{PBW1} for which $0\leq r_k,t_k,s_i<m$ for $i=1,\ldots ,l$, $k=1,\ldots ,D$.

(iv) ${\rm Spec}(Z_0)=\mathbb{C}^{2D}\times(\mathbb{C}^*)^l$ is a complex affine space of dimension equal to ${\rm dim}~\g$.

(v) The subalgebra $Z_0$ is preserved by the action of the braid group automorphisms $T_i$.

(vi) Let $G$ be the connected simply connected Lie group corresponding to the Lie algebra $\g$ and $G^*_0$ the solvable algebraic subgroup in $G\times G$ which consists of elements of the form $(L_+,L_-)\in G\times G$,
$$
(L_+,L_-)=(t,t^{-1})(n_+,n_-),~n_\pm \in N_\pm,~t\in H.
$$

Then ${\rm Spec}(Z_0^0)$ can be naturally identified with the maximal torus $H$ in $G$, and the map
$$
\widetilde{\pi}: {\rm Spec}(Z_0)={\rm Spec}(Z_0^+)\times {\rm Spec}(Z_0^0) \times {\rm Spec}(Z_0^-)\rightarrow G^*_0,
$$
$$
\widetilde{\pi}(u_+,t,u_-)=(t{\bf X}^+(u_+),t^{-1}{\bf X}^-(u_-)^{-1}),~u_\pm\in {\rm Spec}(Z_0^\pm),~t\in {\rm Spec}(Z_0^0),
$$
$$
{\bf X}^\pm: {\rm Spec}(Z_0^\pm) \rightarrow N_\pm,
$$
$$
{\bf X}^-=\exp(x_{\beta_D}^-X_{-\beta_D})\exp(x_{\beta_{D-1}}^-X_{-\beta_{D-1}})\ldots \exp(x_{\beta_1}^-X_{-\beta_1}),
$$
$$
{\bf X}^+=\exp(x_{\beta_D}^+T_0(X_{-\beta_D}))\exp(x_{\beta_{D-1}}^+T_0(X_{-\beta_{D-1}}))\ldots \exp(x_{\beta_1}^+T_0(X_{-\beta_1})),
$$
where $x_{\beta_i}^\pm$ should be regarded as complex-valued functions on ${\rm Spec}(Z_0)$, is an isomorphism of varieties independent of the choice of reduced decomposition of $w_0$.
\end{proposition}

From parts (ii) and (iii) of Proposition \ref{ue} we immediately deduce similar properties for the algebra $U_\varepsilon^{s}({\frak g})$.
\begin{proposition}\label{uq1z}
(i) The subalgebra $\psi_{\{ n_{ij}\}}^{-1}(Z_0)\subset U_\varepsilon^{s}({\frak g})$ is the tensor product of the polynomial algebra with generators $e_{\alpha}^m$, $f_{\alpha}^m$, $\alpha \in \Delta_+$ and of the algebra of Laurent polynomials in $l_i$, $i=1,\ldots,l$.

(ii) $U_\varepsilon^s({\frak g})$ is a free $\psi_{\{ n_{ij}\}}^{-1}(Z_0)$--module with basis the set of monomials $(f)^{\bf r}L^{\bf s}(e)^{\bf t}$ for which $0\leq r_k,t_k,s_i<m$ for $i=1,\ldots ,l$, $k=1,\ldots ,D$.
\end{proposition}

\begin{remark}
In fact ${\rm Spec}(Z_0)$ carries a natural structure of a Poisson--Lie group, and the map $\widetilde{\pi}$ is an isomorphism of algebraic Poisson--Lie groups if $G^*_0$ is regarded as the dual Poisson--Lie group to the Poisson--Lie group $G$ equipped with the standard Sklyanin bracket (see \cite{DKP1}, Theorem 7.6). We shall not need this fact in this paper.
\end{remark}

Let ${\bf K}:{\rm Spec}(Z_0^0)\rightarrow H$ be the map defined by ${\bf K}(h)=h^2$, $h\in H$.

\begin{proposition}\label{LO}{\bf (\cite{DKP1}, Corollary 4.7)}
Let $G^0=N_-HN_+$ be the big cell in $G$. Then the map
$$
\pi={\bf X}^-{\bf K}{\bf X}^+:{\rm Spec}(Z_0)\rightarrow G^0
$$
is independent of the choice of reduced decomposition of $w_0$, and is an unramified covering of degree $2^l$.
\end{proposition}

Denote by $\lambda_0:G^*_0\rightarrow G^0$ the map defined by $\lambda_0(L_+,L_-)=L_-^{-1}L_+$. Then obviously
$\pi=\lambda_0\circ \widetilde{\pi}$.

Define derivations $\underline{x}_i^\pm$ of $U_{\mathcal{A}}(\g)$ by
\begin{equation}\label{qder}
\underline{x}_i^+(u)=\left[ \frac{(X_i^+)^m}{[m]_{q_i}!},u \right],~\underline{x}_i^-(u)=T_0\underline{x}_i^+T_0^{-1}(u),~i=1,\ldots ,l,~u\in U_{\mathcal{A}}(\g).
\end{equation}

Let $\widehat{Z}_0$ be the algebra of formal power series in the $x_{\alpha}^\pm$, $\alpha\in \Delta_+$, and the $l_i^{\pm1}$, $i=1,\ldots ,l$, which define holomorphic functions on ${\rm Spec}(Z_0)=\mathbb{C}^{2D}\times(\mathbb{C}^*)^l$. Let
$$
\widehat{U}_\varepsilon(\g)=U_\varepsilon(\g)\otimes_{Z_0}\widehat{Z}_0,~~
\widehat{Z}_\varepsilon=Z_\varepsilon\otimes_{Z_0}\widehat{Z}_0.
$$

\begin{proposition}{\bf (\cite{DK}, Propositions 3.4, 3.5, \cite{DKP1}, Proposition 6.1, Theorem 6.6)}\label{qcoadj}

(i)On specializing to $q=\varepsilon$, (\ref{qder}) induces a well--defined derivation $\underline{x}_i^\pm$ of ${U}_\varepsilon(\g)$.

(ii)The series
$$
\exp(t\underline{x}_i^\pm)=\sum_{k=0}^\infty \frac{t^k}{k!}(\underline{x}_i^\pm)^k
$$
converge for all $t\in \mathbb{C}$ to a well--defined automorphism of the algebra $\widehat{U}_\varepsilon(\g)$.

(iii)Let $\mathcal{G}$ be the group of automorphisms generated by the one--parameter groups $\exp(t\underline{x}_i^\pm)$, $i=1,\ldots, l$. The action of $\mathcal{G}$ on $\widehat{U}_\varepsilon(\g)$ preserves the subalgebras $\widehat{Z}_\varepsilon$ and $\widehat{Z}_0$, and hence $\mathcal{G}$ acts by holomorphic automorphisms on the complex algebraic varieties ${\rm Spec}(Z_\varepsilon)$ and ${\rm Spec}(Z_0)$.

(iv)Let $\mathcal{O}$ be a conjugacy class in $G$. The intersection $\mathcal{O}^0=\mathcal{O}\bigcap G^0$ is a smooth connected variety, and the variety $\pi^{-1}(\mathcal{O}^0)$ is a $\mathcal{G}$--orbit in ${\rm Spec}(Z_0)$.

(v)If $\mathcal{P}$ is a $\mathcal{G}$--orbit in ${\rm Spec}(Z_0)$ then the connected components of $\tau^{-1}(\mathcal{P})$ are $\mathcal{G}$--orbits in ${\rm Spec}(Z_\varepsilon)$.
\end{proposition}

Given a homomorphism $\eta:{Z}_0 \rightarrow \mathbb{C}$, let
$$
U_\eta(\g)={U}_\varepsilon(\g)/I_\eta,
$$
where $I_\eta$ is the ideal in ${U}_\varepsilon(\g)$ generated by elements $z-\eta(z)$, $z\in Z_0$. By part (iii) of Proposition \ref{ue} $U_\eta(\g)$ is an algebra of dimension $m^{{\rm dim}~\g}$ with linear basis the set of monomials $(X^+)^{\bf r}L^{\bf s}(X^-)^{\bf t}$ for which $0\leq r_k,t_k,s_i<m$ for $i=1,\ldots ,l$, $k=1,\ldots ,D$.

If $\widetilde{g}\in \mathcal{G}$ then for any $\eta\in {\rm Spec}(Z_0)$ we have $\widetilde{g}\eta\in {\rm Spec}(Z_0)$ by part (iii) of Proposition \ref{qcoadj}, and by part (ii) of the same proposition $\widetilde{g}$ induces an isomorphism of algebras,
$$
\widetilde{g}:U_\eta(\g) \rightarrow U_{\widetilde{g}\eta}(\g).
$$


\section{Whittaker vectors}\label{WHITT}

\setcounter{equation}{0}
\setcounter{theorem}{0}

In this section we introduce the notion of Whittaker vectors for modules over quantum groups at roots of unity and prove an analogue of the Engel theorem for them. We start by studying some properties of quantum groups at roots of unity.

From now on we fix an element $s\in W$. Let $\Delta_+$ be a system of positive roots associated to $s^{-1}$. We also fix  positive integer $d$ such that $p_{ij}\in \frac{1}{d}\mathbb{Z}$ for any $i<j$ (or $i>j$), $i,j=1,\ldots ,l$, where the numbers $p_{ij}$ are defined by formula (\ref{pij}). We shall always assume that $m>d_i$ is odd and coprime to all $d_i$, $i=1,\ldots ,l$ and that $d$ and $m$ are coprime. The last condition is equivalent to the existence of an integer $n$ such that $\varepsilon^{nd-1}=1$. We fix an integer valued solution $n_{ij}$ to equations (\ref{eqpi}) and identify the algebra $U_\varepsilon^{s^{-1}}({\frak g})$ associated to the Weyl group element $s^{-1}$ with $U_\varepsilon({\frak g})$ using Theorem \ref{newreal} and the solution $-n_{ij}$ to equations (\ref{eqpi}). Using this identification $U_\varepsilon^{s^{-1}}({\frak m}_-)$ can be regarded as a subalgebra in $U_\varepsilon({\frak g})$. Therefore for every  character $\eta: Z_0 \rightarrow \mathbb{C}$ one can define the corresponding subalgebra in $U_\eta({\frak g})$. We denote this subalgebra by $U_\eta({\frak m}_-)$. By part (ii) of Proposition \ref{uq1z} we have ${\rm dim}~U_\eta({\frak m}_-)=m^{{\rm dim}~\m_-}$

First we study some properties of the finite dimensional algebras $U_\eta({\frak g})$ and $U_\eta({\frak m}_-)$.
In order to define Whittaker vectors for quantum groups at roots of unity we shall need some auxiliary notions that we are going to discuss now.

Below we use the normal ordering of $\Delta_+$ associated to $s^{-1}$.
Observe that by Proposition \ref{pord} for any two roots $\alpha, \beta\in \Delta_{\m_+}$ such that $\alpha<\beta$ the sum $\alpha+\beta$ can not be represented as a linear combination $\sum_{k=1}^qc_k\gamma_{i_k}$, where $c_k\in \mathbb{N}$ and $\alpha<\gamma_{i_1}<\ldots <\gamma_{i_k}<\beta$, and hence from commutation relations (\ref{erel}) one can deduce that
\begin{equation}\label{eqJ}
f_{\alpha}f_{\beta}-\varepsilon^{(\alpha,\beta)-nd(\frac{1+s}{1-s}P_{{\h'}^*}\alpha,\beta)}f_{\beta}f_{\alpha}=\sum_{\alpha<\delta_1<\ldots<\delta_n<\beta}C'(k_1,\ldots,k_n)
f_{\delta_n}^{k_n}f_{\delta_{n-1}}^{k_{n-1}}\ldots f_{\delta_1}^{k_1}\in \mathcal{J},
\end{equation}
where at least one of the roots $\delta_i$ in the right hand side of the last formula belongs to $\Theta=\{\alpha \in \Delta_{\m_+}:\alpha \not\in \{\gamma_1, \ldots ,\gamma_{l'}\}\}$, $\mathcal{J}$ is the ideal in $U_{\eta}({\frak m}_-)$ generated by the elements $f_\beta\in U_{\eta}({\frak m}_-)$, $\beta\in \Theta$. Thus from part (ii) of Proposition \ref{uq1z} and commutation relations (\ref{eqJ}) it follows that if $\delta_1<\delta_2<\ldots<\delta_b$ are the roots in the segment $\Delta_{\m_+}$, the elements
\begin{equation}\label{basJJ}
x_{k_1,\ldots,k_b}=f_{\delta_b}^{k_b}f_{\delta_{b-1}}^{k_{b-1}}\ldots f_{\delta_1}^{k_1}
\end{equation}
for $k_i\in \Bbb{N}$, $k_i<m$ form a linear basis of $U_{\eta}({\frak m}_-)$, and elements (\ref{basJJ}) for $k_i\in \Bbb{N}$, $k_i<m$ and $k_i>0$ for at least one $\delta_i\in \Theta$ form a linear basis of $\mathcal{J}$.

\begin{lemma}\label{Jacob}
Let $\eta$ be an element of ${\rm Spec}(Z_0)$. Assume that $\eta(f_{\gamma_i}^m)=a_i\neq 0$ for $i=1,\ldots ,l'$and that and $\eta(f_\beta^m)=0$ for $\beta\in \Delta_{\m_+}$, $\beta\not\in \{\gamma_1, \ldots ,\gamma_{l'}\}$, and hence $f_{\gamma_i}^m=\eta(f_{\gamma_i}^m)=a_i\neq 0$ in $U_{\eta}({\frak m}_-)$ for $i=1,\ldots ,l'$ and $f_\beta^m=0$ in $U_{\eta}({\frak m}_-)$ for $\beta\in \Delta_{\m_+}$, $\beta\not\in \{\gamma_1, \ldots ,\gamma_{l'}\}$.
Then the ideal $\mathcal{J}$ is the Jacobson radical of $U_{\eta}({\frak m}_-)$ and $U_{\eta}({\frak m}_-)/\mathcal{J}$ is isomorphic to the truncated polynomial algebra $$\mathbb{C}[f_{\gamma_1},\ldots,f_{\gamma_{l'}}]/\{f_{\gamma_i}^m=a_i\}_{ i=1,\ldots ,l'}$$.
\end{lemma}
\begin{proof}
First we show that $\mathcal{J}$ is nilpotent.


Let $i$ be the largest number such that $k_j=0$ for $j>i$ in (\ref{basJJ}) and $k_i\neq 0$. Then we define the degree of $x_{k_1,\ldots,k_b}$ by
$$
deg(x_{k_1,\ldots,k_b})=(k_i,i)\in \{1,\ldots,m-1\}\times \{1,\ldots,b\}.
$$

Equip $\{1,\ldots,m-1\}\times \{1,\ldots,b\}$ with the order such that $(k,i)<(k',j)$ if $j>i$ or $j=i$ and $k'>k$.

For any given $(k,i)\in \{1,\ldots,m-1\}\times \{1,\ldots,b\}$ denote by $(U_{\eta}({\frak m}_-))_{(k,i)}$
 the linear span of the elements $x_{k_1,\ldots,k_b}$ with $deg(x_{k_1,\ldots,k_b})\leq (k,i)$ and define $\mathcal{J}_{(k,i)}=\mathcal{J}\bigcap (U_{\eta}({\frak m}_-))_{(k,i)}$.
We also have $(U_{\eta}({\frak m}_-))_{(k,i)}\subset (U_{\eta}({\frak m}_-))_{(k',j)}$ and
$\mathcal{J}_{(k,i)}\subset \mathcal{J}_{(k',j)}$ if $(k,i)<(k',j)$, and $\mathcal{J}_{(m-1,b)}=\mathcal{J}$. Note that for the first few $i$ linear spaces $\mathcal{J}_{(k,i)}$ may be trivial, and these are all possibilities when those spaces can be trivial.

We shall prove that $\mathcal{J}$ is nilpotent by induction over the order in $\{1,\ldots,m-1\}\times \{1,\ldots,b\}$.
Let $(k,i)$ be minimal possible such that $\mathcal{J}_{(k,i)}$ is not trivial. Then we must have $k=1$. If $y\in \mathcal{J}_{(1,i)}$ then $y$ must be of the form
\begin{equation}\label{yf}
y=f_{\beta}v,
\end{equation}
where $v$ is a linear combination of elements of the form $f_{\delta_{i_1}}^{k_{1}}\ldots f_{\delta_{i_r}}^{k_r}$ for ${\delta_{i_1}},\ldots ,{\delta_{i_r}}\in \{\gamma_{n+1},\ldots , \gamma_{l'}\}$, $\beta>\delta_{i_1}$, and $\beta$ is the first root from the set $\Theta$ greater than $\gamma_{l'}$ in the normal ordering of $\Delta_+$  associated to $s^{-1}$. Here it is assumed that $f_{\delta_{1}}^{k_{1}}\ldots f_{\delta_r}^{k_r}=1$ if the set $\{\gamma_{n+1},\ldots , \gamma_{l'}\}$ is empty.

Now equation (\ref{eqJ}) implies that for any $f_{\delta_{i_j}}$ which appears in the expression for $v$ one has
\begin{equation}\label{rr1}
f_{\beta}f_{\delta_{i_j}}-\varepsilon^{(\beta,\delta_{i_j})-nd(\frac{1+s}{1-s}P_{{\h'}^*}\beta,\delta_{i_j})}f_{\delta_{i_j}}f_{\beta}\in \mathcal{J}_{(m-1,i-1)}=0
\end{equation}
as by our choice of $i$ $\mathcal{J}_{(m-1,i-1)}=0$.

Formula (\ref{rr1}) implies that the product of $m$ elements of type (\ref{yf}) can be represented in the form
$$
f_{\beta}^{m}v',
$$
where $v'$ is of the same form as $v$. Since $f_{\beta}^{m}=0$ we deduce that $\mathcal{J}_{(1,i)}^m=0$.

Now assume that $\mathcal{J}_{(k,i)}^K=0$ for some $K>0$. Let $(k',i')$ be the smallest element of $\{1,\ldots,m-1\}\times \{1,\ldots,b\}$ which satisfies $(k,i)<(k',i')$. Then by Propositions \ref{rootss} and \ref{uq1z}, and by (\ref{eqJ}), any element of  $\mathcal{J}_{(k',i')}$ is of the form $f_{\delta_{i'}}u+u'$, where $u'\in \mathcal{J}_{(k,i)}$ and if $\delta_{i'}\in \Theta$ then $u\in (U_{\eta}({\frak m}_-))_{(k,i)}$; if $\delta_{i'}\not\in \Theta$ then $u\in \mathcal{J}_{(k,i)}$.

Now equation (\ref{eqJ}) together with (\ref{erel}) imply that for any $u\in (U_{\eta}({\frak m}_-))_{(k,i)}$ one has
\begin{equation}\label{ufc}
uf_{\delta_{i'}}=cf_{\delta_{i'}}u+w,
\end{equation}
where $c$ is a non--zero constant depending on $u$, and $w\in \mathcal{J}_{(k,i)}$.
By formula (\ref{ufc}) the product of $m$ elements $f_{\delta_{i'}}u_p+u'_p$, $p=1,\ldots,m$ of the type described above can be represented in the form
\begin{equation}\label{sj}
\sum_{j=0}^mf_{\delta_{i'}}^jc_j,
\end{equation}
where $c_j\in \mathcal{J}_{(k,i)}$ for $j=0,\ldots ,m-1$ and if $\delta_{i'}\in \Theta$ then $c_m\in (U_{\eta}({\frak m}_-))_{(k,i)}$; if $\delta_{i'}\not\in \Theta$ then $c_m\in \mathcal{J}_{(k,i)}$. In the former case $f_{\delta_{i'}}^m=0$, and the last term in sum (\ref{sj}) is zero; in the latter case $f_{\delta_{i'}}^m=\eta(f_{\delta_{i'}}^m)\neq 0$, and the last term in sum (\ref{sj}) is from $\mathcal{J}_{(k,i)}$. So we can combine it with the term corresponding to $j=0$. In both cases sum (\ref{sj}) takes the form
\begin{equation}\label{st}
\sum_{j=0}^{m-1}f_{\delta_{i'}}^jc_j',
\end{equation}
where $c_j'\in \mathcal{J}_{(k,i)}$. By (\ref{ufc}) the product of $K$ sums of type (\ref{st}) is of the form
$$
\sum_{j=0}^{(m-1)K}f_{\delta_{i'}}^jc_j'',
$$
where each $c_j''$ is a linear combination of elements from $\mathcal{J}_{(k,i)}^K$.
By our assumption $\mathcal{J}_{(k,i)}^K=0$, and hence the product of any $mK$ elements of $\mathcal{J}_{(k',i')}$ is zero. This justifies the induction step and proves that $\mathcal{J}_{(m-1,b)}=\mathcal{J}$ is nilpotent.
Hence $\mathcal{J}$ is contained in the Jacobson radical of $U_{\eta}({\frak m}_-)$.

Using commutation relations (\ref{erel}) we also have (see the proof of Theorem \ref{qnil})
$$
f_{\gamma_i}f_{\gamma_j} - f_{\gamma_j}f_{\gamma_i}\in \mathcal{J}.
$$
Therefore the quotient algebra $U_{\eta}({\frak m}_-)/\mathcal{J}$ is isomorphic to the truncated polynomial algebra $$\mathbb{C}[f_{\gamma_1},\ldots,f_{\gamma_{l'}}]/\{f_{\gamma_i}^m=a_i\}_{ i=1,\ldots ,l'}$$ which is semisimple. Therefore $\mathcal{J}$ coincides with the Jacobson radical of $U_{\eta}({\frak m}_-)$.
\end{proof}



In Theorem \ref{qnil} we constructed some characters of the algebra $U_\varepsilon^{s}({\frak m}_-)$. Similarly one can define characters of the algebra $U_\varepsilon^{s^{-1}}({\frak m}_-)$. Now we show that the algebra $U_\eta({\frak m}_-)$ has a finite number of irreducible representations which are one--dimensional, and all those representations can be obtained from each other by twisting with the help of automorphisms of $U_\eta({\frak m}_-)$.

\begin{proposition}\label{irrepod}
Let $\eta$ be an element of ${\rm Spec}(Z_0)$. Assume that $\eta(f_{\gamma_i}^m)=a_i\neq 0$ for $i=1,\ldots ,l'$and that and $\eta(f_\beta^m)=0$ for $\beta\in \Delta_{\m_+}$, $\beta\not\in \{\gamma_1, \ldots ,\gamma_{l'}\}$, and hence $f_{\gamma_i}^m=\eta(f_{\gamma_i}^m)=a_i\neq 0$ in $U_{\eta}({\frak m}_-)$ for $i=1,\ldots ,l'$ and $f_\beta^m=0$ in $U_{\eta}({\frak m}_-)$ for $\beta\in \Delta_{\m_+}$, $\beta\not\in \{\gamma_1, \ldots ,\gamma_{l'}\}$. Then all non--zero irreducible representations of the algebra $U_\eta({\frak m}_-)$ are one--dimensional and have the form
\begin{equation}\label{chichar}
\chi(f_\beta)=\left\{ \begin{array}{ll} 0 & \beta \not \in \{\gamma_1, \ldots, \gamma_{l'}\} \\ c_i & \beta=\gamma_i,~i=1,\ldots ,l'
\end{array}
\right  .,
\end{equation}
where complex numbers $c_i$ satisfy the conditions $c_i^m=a_i$, $i=1,\ldots ,l'$.
Moreover, all non--zero irreducible representations of $U_\eta({\frak m}_-)$ can be obtained from each other by twisting with the help of automorphisms of $U_\eta({\frak m}_-)$.
\end{proposition}

\begin{proof}
Let $V$ be a non--zero finite--dimensional irreducible $U_{\eta}({\frak m}_-)$--module.
By Corollary 54.13 in \cite{CR} elements of the ideal $\mathcal{J}\subset U_{\eta}({\frak m}_-)$ act by zero transformations on $V$. Hence $V$ is in fact an irreducible representation of the algebra $U_{\eta}({\frak m}_-)/\mathcal{J}$ which is isomorphic to the truncated polynomial algebra $$\mathbb{C}[f_{\gamma_1},\ldots,f_{\gamma_{l'}}]/\{f_{\gamma_i}^m=a_i\}_{ i=1,\ldots ,l'}.$$ The last algebra is commutative and all its complex irreducible representations are
one--dimensional. Therefore $V$ is one--dimensional, and if $v$ is a nonzero element of $V$ then $f_{\gamma_i}v=c_iv$, for some $c_i\in \Bbb{C}$, $i=1,\ldots ,l'$.
Note that $\eta(f_{\gamma_i}^m)=a_i\neq 0$, $i=1,\ldots ,l'$ and hence $c_i^m=a_i\neq 0$, $i=1,\ldots ,l'$. In particular, the elements $f_{\gamma_i}$ act on $V$ by semisimple automorphisms.

If we denote by $\chi:U_{\eta}({\frak m}_-) \rightarrow \mathbb{C}$ the character of $U_{\eta}({\frak m}_-)$ such that
$$
\chi(f_\beta)=\left\{ \begin{array}{ll} 0 & \beta \not \in \{\gamma_1, \ldots, \gamma_{l'}\} \\ c_i & \beta=\gamma_i, ~i=1,\ldots ,l'
\end{array}
\right  .
$$
and by $\mathbb{C}_{\chi}$ the corresponding one--dimensional representation of $U_{\eta}({\frak m}_-)$ then we have $V=\mathbb{C}_{\chi}$.

Now we have to prove that the representations $\mathbb{C}_{\chi}$ for different characters $\chi$ are obtained from each other by twisting with the help of automorphisms of $U_{\eta}({\frak m}_-)$.

Since $c_i^m=a_i$, $i=1,\ldots ,l'$ there are only finitely many possible characters $\chi$ corresponding to the given $\eta$ in the statement of this proposition.
If $\chi$ and $\chi'$ are two such characters, $\chi(f_{\gamma_i})=c_i$, $i=1,\ldots ,l'$ and $\chi'(f_{\gamma_i})=c_i'$, $i=1,\ldots ,l'$ then the relations $c_i^m={c_i'}^m=a_i$, $i=1,\ldots ,l'$ imply that
$c_i'=\varepsilon^{m_i}c_i$, $0\leq m_i \leq m-1$, $m_i\in \mathbb{Z}$, $i=1,\ldots ,l'$.

Now observe that for any $h\in \h$ the map defined by $f_\alpha \mapsto \varepsilon^{\alpha(h)}f_\alpha$, $\alpha \in \Delta_{\m_+}$ is an automorphism of the algebra $U_\varepsilon^{s^{-1}}({\frak m}_-)$ generated by elements $f_\alpha$, $\alpha \in \Delta_{\m_+}$ with defining relations (\ref{erel}). Here the principal branch of the analytic function $\varepsilon^z$ is used to define $\varepsilon^{\alpha(h)}$, so that $\varepsilon^{\alpha(h)}\varepsilon^{\beta(h)}=\varepsilon^{(\alpha+\beta)(h)}$ for any $\alpha,\beta\in \Delta_{\m_+}$. If in addition $\varepsilon^{m\gamma_i(h)}=1$, $i=1,\ldots ,l'$ the above defined map gives rise to an automorphism $\varsigma$ of $U_{\eta}({\frak m}_-)$. Indeed in that case $(\varepsilon^{\gamma_i(h)}f_{\gamma_i})^m=f_{\gamma_i}^m$, $i=1,\ldots ,l'$ and all the remaining defining relations $f_{\gamma_i}^m=\eta(f_{\gamma_i}^m)=a_i\neq 0$, $i=1,\ldots ,l'$, $f_{\beta}^m=\eta(f_{\beta}^m)=0$, $\beta \in \Delta_{\m_+}$, $\beta \not \in \{\gamma_1, \ldots, \gamma_{l'}\}$ of the algebra $U_{\eta}({\frak m}_-)$ are preserved by the action of the above defined map $\varsigma$.

Now fix $h\in \h$ such that $\gamma_i(h)=m_i$, $i=1,\ldots ,l'$. Obviously we have $\varepsilon^{mm_i}=1$, $i=1,\ldots ,l'$. We claim that the representation $\mathbb{C}_{\chi}$ twisted by the corresponding automorphism  $\varsigma$ coincides with $\mathbb{C}_{\chi'}$. Indeed, we obtain
$$
\chi(\varsigma f_{\gamma_i})=\chi(\varepsilon^{m_i}f_{\gamma_i})=\varepsilon^{m_i}c_i=c_i',~i=1,\ldots ,l'.
$$
This completes the proof of the proposition.
\end{proof}

Let $V$ be a $U_{\eta}({\frak g})$--module, where $\eta$ is an element of ${\rm Spec}(Z_0)$ such that $\eta(f_{\gamma_i}^m)=a_i\neq 0$ for $i=1,\ldots ,l'$and that and $\eta(f_\beta^m)=0$ for $\beta\in \Delta_{\m_+}$, $\beta\not\in \{\gamma_1, \ldots ,\gamma_{l'}\}$. Let $\chi:U_\eta({\frak m}_-)\rightarrow \mathbb{C}$ be a character defined in the  Proposition \ref{irrepod}, $\mathbb{C}_{\chi}$ the corresponding one--dimensional $U_\eta({\frak m}_-)$--module. Then the space $V_\chi={\rm Hom}_{U_\eta({\frak m}_-)}(\mathbb{C}_{\chi},V)$ is called the space of Whittaker vectors of $V$. Elements of $V_\chi$ are called Whittaker vectors.

Now we describe the space of Whittaker vectors in terms of a nilpotent action of the unital subalgebra $U_{\eta_0}({\frak m}_-)$ generated by $f_\alpha$, $\alpha\in \Delta_{\m_+}$ in the small quantum group $U_{\eta_0}({\frak g})=U_\varepsilon^{s^{-1}}(\g)/I_{\eta_0}$ corresponding to the trivial central character $\eta_0$ such that $\widetilde{\pi}(\eta_0)=1\in G_0^*$  and $\eta_0(x_\alpha^\pm)=0$, $\alpha\in \Delta_+$, $\eta_0(l_i)=1$, $i=1,\ldots , l$.
Recall that the algebra $U_\varepsilon^{s^{-1}}(\g)$ is a Hopf algebra. 
We shall need the following formula for the action of the comultiplication on the quantum root elements $f_\beta$, $\beta=\sum_{i=1}^lc_i\alpha_i\in \Delta_{\m_+}$, $c_i\in \mathbb{N}$,
\begin{eqnarray}\label{cm1}
\Delta_{s^{-1}}(f_{\beta})=\prod_{i=1}^lK_i^{c_i}\prod_{i,j=1}^lL_j^{\frac{nd}{d_j}(\frac{1+s}{1-s}P_{{\h'}^*}\alpha_i,\alpha_j)c_i}\otimes f_\beta+f_\beta \otimes 1+ \\
+\sum_i y_i\otimes x_i,~x_i\in U_{<\beta}, y_i\in U_{>\beta}U_\varepsilon^{s^{-1}}(\h), \nonumber
\end{eqnarray}
where $U_{<\beta}$ is the subalgebra (without unit) in $U_\varepsilon^{s^{-1}}({\frak m}_-)$ generated by ${f}_{\alpha}$, $\alpha<\beta$ and $U_{>\beta}$ is the subalgebra (without unit) in $U_\varepsilon^{s^{-1}}(\n_-)$ generated by ${f}_{\alpha}$, $\alpha>\beta$.

To derive formula (\ref{cm1}) we first observe that by Corollary 4.3.2 in \cite{Dr} one has in $U_h^{s^{-1}}(\g)$
\begin{eqnarray}\label{ddq}
\Delta(X_{\beta}^-)=\prod_{i=1}^lK_i^{c_i}\otimes X_\beta^-+X_\beta^- \otimes 1+ \\
+\sum_i y_i'\otimes x_i', \nonumber
\end{eqnarray}
where $x_i'$ are elements of the subalgebra (without unit) in $U_h(\g)$ generated by ${X}_{\alpha}^-$, $\alpha<\beta$ and $y_i'$ are elements of the subalgebra (without unit) in $U_h(\g)$ generated by ${X}_{\alpha}^-$, $\alpha>\beta$ and by the $L_j^{\pm 1}, j=1,\ldots ,l$.

Using formula (\ref{defds}), the definition of the elements $f_\beta$ and the action of the comultiplication $\Delta$ on the elements $H_i$ we obtain from formula (\ref{ddq}) that
\begin{eqnarray*}
\Delta_{s^{-1}}(f_{\beta})=\prod_{i=1}^lK_i^{c_i}\prod_{i,j=1}^lL_j^{\frac{nd}{d_j}(\frac{1+s}{1-s}P_{{\h'}^*}\alpha_i,\alpha_j)c_i}\otimes f_\beta+f_\beta \otimes 1+ \\
+\sum_i y_i\otimes x_i,
\end{eqnarray*}
where $x_i$ are elements of the subalgebra (without unit) in $U_h^{s^{-1}}(\g)$ generated by $f_{\alpha}$, $\alpha<\beta$ and by the $L_j^{\pm 1}, j=1,\ldots ,l$ and $y_i$ are elements of the subalgebra (without unit) in $U_h^{s^{-1}}(\g)$ generated by $f_{\alpha}$, $\alpha>\beta$ and by the $L_j^{\pm 1}, j=1,\ldots ,l$.

On the other hand formula (\ref{comults}) and commutation relations (\ref{erel}) imply that in fact $x_i$ are elements of the subalgebra (without unit) in $U_h^{s^{-1}}(\g)$ generated by $f_{\alpha}$, $\alpha<\beta$. Since $U_{\mathcal{A}}^{s^{-1}}(\g)$ is in fact a Hopf subalgebra in $U_h^{s^{-1}}(\g)$ we deduce that $x_i$ are elements of the subalgebra (without unit) in $U_{\mathcal{A}}^{s^{-1}}(\g)$ generated by $f_{\alpha}$, $\alpha<\beta$ and $y_i$ are elements of the subalgebra (without unit) in $U_{\mathcal{A}}^{s^{-1}}(\g)$ generated by $f_{\alpha}$, $\alpha>\beta$ and by the $L_j^{\pm 1}, j=1,\ldots ,l$. This implies (\ref{cm1}).

Formula (\ref{cm1}) shows that $U_\varepsilon^{s^{-1}}({\frak m}_-)$ is a right coideal in $U_\varepsilon^{s^{-1}}(\g)$. One can also equip the algebra $U_\varepsilon^{s^{-1}}({\frak m}_-)$ with a character given by formula (\ref{char}), where the numbers $c_i$ are the same as in the definition of the character $\chi$. We denote this character by the same letter, $\chi: U_\varepsilon^{s^{-1}}({\frak m}_-)\rightarrow \mathbb{C}$.

Note that $V$ can be regarded as a $U_\varepsilon(\g)$--module and a $U_\varepsilon^{s^{-1}}(\g)$--module assuming that the ideal $I_\eta$ acts on $V$ in the trivial way.
Now observe that $\Delta_{s^{-1}}:U_\varepsilon^{s^{-1}}({\frak m}_-)\rightarrow U_\varepsilon^{s^{-1}}(\g)\otimes U_\varepsilon^{s^{-1}}({\frak m}_-)$ is a homomorphism of algebras. Composing it with the tensor product $S_{s^{-1}}\otimes \chi$ of the anti--homomorphism $S_{s^{-1}}$ and of the character $\chi$, which can be regarded as an anti--homomorphism as well, one can define an anti--homomorphism, $U_\varepsilon^{s^{-1}}({\frak m}_-)\rightarrow U_\varepsilon^{s^{-1}}(\g)$, $x\mapsto S_{s^{-1}}(x_1)\chi(x_2)$, $\Delta_{s^{-1}}(x)=x_1\otimes x_2$, $x\in U_\varepsilon^{s^{-1}}({\frak m}_-)$.

Using this anti--homomorphism one can introduce a right $U_\varepsilon^{s^{-1}}({\frak m}_-)$--action on $V$ which we call the adjoint action and denote it by $\rm Ad$. It is given by the formula (compare with the definition of the adjoint action in (\ref{Adact}))
\begin{equation}\label{AD}
{\rm Ad}~x v=S_{s^{-1}}(x_1)\chi(x_2)v, x\in U_\varepsilon^{s^{-1}}({\frak m}_-), v\in V,
\end{equation}
where $\Delta_{s^{-1}}(x)=x_1\otimes x_2$. 

Note that using the Swedler notation for the comultiplication, $(\Delta_{s^{-1}}\otimes id \otimes id) (\Delta_{s^{-1}}\otimes id) \Delta_{s^{-1}}(x)=x_1\otimes x_2\otimes x_3 \otimes x_4$, the coassiciativity of the comultiplication and the definition of the antipode we have for any $x \in U_\varepsilon^{s^{-1}}({\frak m}_-)$, $y\in U_\varepsilon^{s^{-1}}(\g)$, $v\in V$ (compare with the proof of Lemma 2.2 in \cite{JL})
\begin{equation}\label{adm}
{\rm Ad}x(yv)=S_{s^{-1}}(x_1)\chi(x_2)yv=S_{s^{-1}}(x_1)yx_2S_{s^{-1}}(x_3)\chi(x_4)v={\rm Ad}x_1(y){\rm Ad}x_2(v).
\end{equation}

Similarly to the Proposition in Section 5.6 in \cite{DKP1} we infer that $Z_0$ is a Hopf subalgebra in $U_\varepsilon^{s^{-1}}(\g)$. Namely,
\begin{eqnarray*}
\Delta_{s^{-1}}(f_i^m)=K_i^{m}\prod_{j=1}^lL_j^{m\frac{nd}{d_j}(\frac{1+s}{1-s}P_{{\h'}^*}\alpha_i,\alpha_j)}\otimes f_i^m+f_i^m \otimes 1,  \\
\Delta_{s^{-1}}(e_i^m)=e_i^m\otimes K_i^{-m}+\prod_{j=1}^lL_j^{-m\frac{nd}{d_j}(\frac{1+s}{1-s}P_{{\h'}^*}\alpha_i,\alpha_j)}\otimes e_i^m, \\
\Delta_{s^{-1}}(L_i^m)=L_i^m\otimes L_i^m.
\end{eqnarray*}
Therefore recalling that by the definition of $\chi$
for $x\in U_\varepsilon^{s^{-1}}({\frak m}_-)\bigcap Z_0$ one has $\chi(x)=\eta(x)$ we deduce
$$
{\rm Ad}~x v=S_{s^{-1}}(x_1)\chi(x_2)v=\eta(S_{s^{-1}}(x_1)x_2)v=\varepsilon_{s^{-1}}(x)v, v\in V,
$$
where $\varepsilon_{s^{-1}}$ is the counit of $U_\varepsilon^{s^{-1}}(\g)$. 
Note that by the definition of the ideal $I_{\eta_0}$ the ideal $U_\varepsilon^{s^{-1}}({\frak m}_-)\bigcap I_{\eta_0}\subset U_\varepsilon^{s^{-1}}({\frak m}_-)$  is generated by the elements $f_\alpha^m$, $\alpha\in \Delta_{\m_+}$ and  $\varepsilon_{s^{-1}}(f_\alpha^m)=0$ for $\alpha\in \Delta_{\m_+}$ by the definition of $\varepsilon_{s^{-1}}$.
Hence the adjoint action of $U_\varepsilon^{s^{-1}}({\frak m}_-)$ on $V$ induces an action of the subalgebra $U_{\eta_0}({\frak m}_-)$ of the small quantum group $U_{\eta_0}({\frak g})$. We call this action the adjoint action as well.

Note that the small quantum group $U_{\eta_0}({\frak g})$ is a Hopf algebra with the comultiplication inherited from $U_\varepsilon^{s^{-1}}({\frak g})$.
Now arguments used in the proof of Proposition 5.6 in \cite{S11} can be applied verbatim to establish the following lemma.

\begin{lemma}
The space of Whittaker vectors $V_\chi$ coincides with the space of $U_{\eta_0}({\frak m}_-)$--invariants for the adjoint action on $V$,
\begin{equation}\label{ADW}
V_\chi=\{ v\in V: {\rm Ad}~x(v)=\varepsilon_{s^{-1}}(x)v ~~~\forall x\in U_{\eta_0}({\frak m}_-)\}.
\end{equation}
\end{lemma}

\begin{proof}
Indeed, denote by $T_\beta$ the factor $\prod_{i=1}^lK_i^{c_i}\prod_{i,j=1}^lL_j^{\frac{nd}{d_j}(\frac{1+s}{1-s}P_{{\h'}^*}\alpha_i,\alpha_j)c_i}$ which appears in (\ref{cm1}), $T_\beta=\prod_{i=1}^lK_i^{c_i}\prod_{i,j=1}^lL_j^{\frac{nd}{d_j}(\frac{1+s}{1-s}P_{{\h'}^*}\alpha_i,\alpha_j)c_i}$. Then by the definition of the antipode we have from (\ref{cm1}) 
$$
S_{s^{-1}}(T_\beta)f_\beta+S_{s^{-1}}(f_\beta)
+\sum_i S_{s^{-1}}(y_i)x_i=\varepsilon_{s^{-1}}(f_\beta)=0.
$$
Since $S_{s^{-1}}(T_\beta)=T_\beta^{-1}=\prod_{i=1}^lK_i^{-c_i}\prod_{i,j=1}^lL_j^{-\frac{nd}{d_j}(\frac{1+s}{1-s}P_{{\h'}^*}\alpha_i,\alpha_j)c_i}$ this yields
\begin{equation}\label{cm2}
S_{s^{-1}}(f_\beta)=-S_{s^{-1}}(T_\beta)f_\beta
-\sum_i S_{s^{-1}}(y_i)x_i.
\end{equation}

Now for $\beta\in \Delta_{\m_+}$, (\ref{cm1}), (\ref{cm2}) and definition (\ref{AD}) of the adjoint action imply
\begin{eqnarray}\label{cm3}
{\rm Ad}~f_{\beta}v= T_\beta^{-1}\chi(f_\beta)v-T_\beta^{-1}f_\beta v - 
\sum_i S_{s^{-1}}(y_i) x_i v +\sum_i S_{s^{-1}}(y_i) \chi(x_i)v= \nonumber \\
=T_\beta^{-1}(\chi(f_\beta)-f_\beta) v + 
\sum_i S_{s^{-1}}(y_i) (\chi(x_i)-x_i)v, ~x_i\in U_{<\beta}, y_i\in U_{>\beta}U_\varepsilon^{s^{-1}}(\h).
\end{eqnarray}

If $v\in V_\chi$ we immediately obtain from (\ref{cm3}) that ${\rm Ad}~f_{\beta}v=0$ for any $\beta\in \Delta_{\m_+}$, i.e. $v$ belongs to the right hand side of (\ref{ADW}).

Conversely, suppose that $v$ belongs to the right hand side of (\ref{ADW}). We shall show that $xv=\chi(x)v$ for any $x\in U_\varepsilon^{s^{-1}}({\frak m}_-)$. Let $\overline{U}_{<\beta}$ be the subalgebra with unit generated by $U_{<\beta}$.
We proceed by induction over the subalgebras $\overline{U}_{<\delta_k}$, $k=1,\ldots b+1$, where as before $\delta_1<\ldots <\delta_b$ is the normally ordered segment $\Delta_{\m_+}$ and we define $\overline{U}_{<\delta_{b+1}}$ to be the subalgebra $U_\varepsilon^{s^{-1}}({\frak m}_-)$. 

Observe that $\delta_1$  is a simple root and hence $U_{<\delta_1}=0$. Therefore we deduce from (\ref{cm3}) for $\beta=\delta_1$
$$
{\rm Ad}~f_{\delta_1}v= 
T_{\delta_1}^{-1}(\chi(f_{\delta_1})-f_{\delta_1}) v=0.
$$
Since $T_{\delta_1}^{-1}$ acts on $V$ by an invertible transformation this implies $(\chi(f_{\delta_1})-f_{\delta_1}) v=0$, and hence $xv=\chi(x)v$ for any $x\in \overline{U}_{<\delta_2}$ as $\overline{U}_{<\delta_2}$ is generated by $f_{\delta_1}$.

Now assume that for some $k\leq b$ $xv=\chi(x)v$ for any $x\in \overline{U}_{<\delta_k}$. Then by (\ref{cm3})
$$
{\rm Ad}~f_{\delta_k}v= 
T_{\delta_k}^{-1}(\chi(f_{\delta_k})-f_{\delta_k}) v=0.
$$
As above this implies 
\begin{equation}\label{cm4}
(\chi(f_{\delta_k})-f_{\delta_k}) v=0.
\end{equation}

By Proposition \ref{rootss} any element $y\in \overline{U}_{<\delta_{k+1}}$ can be uniquely represented in the form $y=f_{\delta_k}y'+y''$, where $y',y''\in \overline{U}_{<\delta_k}$. Now by (\ref{cm4}) and by the induction assumption
$$
yv=(f_{\delta_k}y'+y'')v=\chi(f_{\delta_k})\chi(y')v+\chi(y'')v=\chi(y)v,
$$
i.e. $yv=\chi(y)v$ for any $x\in \overline{U}_{<\delta_{k+1}}$. This establishes the induction step and completes the proof.

\end{proof}

The following proposition is an analogue of Engel theorem for quantum groups at roots of unity.

\begin{proposition}\label{Whitt}
Let $\eta$ be an element of ${\rm Spec}(Z_0)$. Assume that $\eta(f_{\gamma_i}^m)=a_i\neq 0$ for $i=1,\ldots ,l'$ and that  $\eta(f_\beta^m)=0$ for $\beta\in \Delta_{\m_+}$, $\beta\not\in \{\gamma_1, \ldots ,\gamma_{l'}\}$, and hence $f_{\gamma_i}^m=\eta(f_{\gamma_i}^m)=a_i\neq 0$ in $U_{\eta}({\frak m}_-)$ for $i=1,\ldots ,l'$ and $f_\beta^m=0$ in $U_{\eta}({\frak m}_-)$ for $\beta\in \Delta_{\m_+}$, $\beta\not\in \{\gamma_1, \ldots ,\gamma_{l'}\}$. Let $\chi:U_\eta({\frak m}_-)\rightarrow \mathbb{C}$ be any character defined in Proposition \ref{irrepod}.
Then any non--zero finite--dimensional $U_{\eta}({\frak g})$--module contains a non--zero Whittaker vector.
\end{proposition}

\begin{proof}
First we show that the augmentation ideal $\mathcal{J}^0$ of $U_{\eta_0}({\frak m}_-)$ coincides with its Jacobson radical  which is nilpotent. The proof of this fact is similar to that of Lemma \ref{Jacob}, and we shall keep the notation used in that proof. 

We define $\mathcal{J}_{(k,i)}^0=\mathcal{J}^0\bigcap (U_{\eta_0}({\frak m}_-))_{(k,i)}$, so that
$\mathcal{J}_{(k,i)}^0\subset \mathcal{J}_{(k',j)}^0$ if $(k,i)<(k',j)$, and $\mathcal{J}_{(m-1,b)}^0=\mathcal{J}^0$. 

We shall prove that $\mathcal{J}^0$ is nilpotent by induction over the order in $\{1,\ldots,m-1\}\times \{1,\ldots,b\}$.
Note that $(k,i)=(1,1)$ is minimal possible such that $\mathcal{J}_{(k,i)}$ is not trivial. If $y\in \mathcal{J}_{(1,1)}$ then $y$ must be of the form
\begin{equation}\label{yf1}
y=af_{l'}, a\in \mathbb{C}.
\end{equation}

The product of $m$ elements of type (\ref{yf1}) is equal to zero,
$$
f_{\beta}^{m}a^m=0,
$$
as $f_{\beta}^{m}=0$. We deduce that $(\mathcal{J}_{(1,1)}^0)^m=0$.

Now assume that $(\mathcal{J}_{(k,i)}^0)^K=0$ for some $K>0$. Let $(k',i')$ be the smallest element of $\{1,\ldots,m-1\}\times \{1,\ldots,b\}$ which satisfies $(k,i)<(k',i')$. Then by Propositions \ref{rootss} and \ref{uq1z} any element of  $\mathcal{J}_{(k',i')}^0$ is of the form $f_{\delta_{i'}}u+u'$, where $u'\in \mathcal{J}_{(k,i)}^0$ and $u\in (U_{\eta_0}({\frak m}_-))_{(k,i)}$.

Now equation (\ref{erel}) implies that for any $u\in (U_{\eta_0}({\frak m}_-))_{(k,i)}$ one has
\begin{equation}\label{ufc1}
uf_{\delta_{i'}}=cf_{\delta_{i'}}u+w,
\end{equation}
where $c$ is a non--zero constant depending on $u$, and $w\in \mathcal{J}_{(k,i)}^0$.
By the formula (\ref{ufc1}) the product of $m$ elements $f_{\delta_{i'}}u_p+u'_p$, $p=1,\ldots,m$ of the type described above can be represented in the form
\begin{equation}\label{sj1}
\sum_{j=0}^mf_{\delta_{i'}}^jc_j,
\end{equation}
where $c_j\in \mathcal{J}_{(k,i)}^0$ for $j=0,\ldots ,m-1$ and $c_m\in (U_{\eta_0}({\frak m}_-))_{(k,i)}$. Since $f_{\delta_{i'}}^m=0$ the last term in sum (\ref{sj1}) is zero. So sum (\ref{sj1}) takes the form
\begin{equation}\label{st1}
\sum_{j=0}^{m-1}f_{\delta_{i'}}^jc_j',
\end{equation}
where $c_j'\in \mathcal{J}_{(k,i)}^0$. By (\ref{ufc1}) the product of $K$ sums of type (\ref{st1}) is of the form
$$
\sum_{j=0}^{(m-1)K}f_{\delta_{i'}}^jc_j'',
$$
where each $c_j''$ is a linear combination of elements from $(\mathcal{J}_{(k,i)}^0)^K$.
By our assumption $(\mathcal{J}_{(k,i)}^0)^K=0$, and hence the product of any $mK$ elements of $\mathcal{J}_{(k',i')}^0$ is zero. This justifies the induction step and proves that $\mathcal{J}_{(m-1,b)}^0=\mathcal{J}^0$ is nilpotent.
Hence $\mathcal{J}^0$ is contained in the Jacobson radical of $U_{\eta_0}({\frak m}_-)$.

The quotient algebra $U_{\eta_0}({\frak m}_-)/\mathcal{J}^0$ is isomorphic to $\mathbb{C}$. Therefore $\mathcal{J}^0$ coincides with the Jacobson radical of $U_{\eta_0}({\frak m}_-)$.

Now let $V$ be a finite--dimensional $U_{\eta}({\frak g})$--module. Then $V$ is also a finite--dimensional  $U_{\eta_0}({\frak m}_-)$--module with respect to the adjoint action. Thus $V$ must contain a non--trivial irreducible $U_{\eta_0}({\frak m}_-)$--submodule with respect to the adjoint action on which the Jacobson radical $\mathcal{J}^0$ must act trivially. From (\ref{ADW}) it follows that this non--trivial irreducible submodule consists of Whittaker vectors.This completes the proof.

\end{proof}

Let
$$
G=\bigcup_{\mathcal{C}\in C(W)}G_\mathcal{C}.
$$
be the Lusztig partition of $G$ (see \cite{L1}; we use the notation of \cite{S10}, Section 4). Here $C(W)\subset \underline{W}$ is a certain subset of the set of conjugacy classes $\underline{W}$ in $W$.

From the above discussion and Proposition 6.2 in \cite{S10} we deduce the following theorem.
\begin{theorem}\label{DKPconj}
Let $\eta \in {\rm Spec}(Z_0)$ be an element  such that $\pi\eta\in G_\mathcal{C}$, $\mathcal{C}\in C(W)$ and $s\in \mathcal{C}$.  Denote by $d$ the number corresponding to $s^{-1}$ and defined in Proposition 6.2 in \cite{S10}. Assume that $m$ and $d$ are coprime.

Then there is a system of positive roots
$\Delta_+^{s^{-1}}$ and a quantum coadjoint transformation $\widetilde{g}$ such that $\xi =\widetilde{g}\eta$ satisfies $\xi(f_{\gamma_i}^m)=a_i\neq 0$ for $i=1,\ldots ,l'$ and $\xi(f_\beta^m)=0$  for $\beta\in \Delta_{\m_+}$, $\beta\not\in \{\gamma_1, \ldots ,\gamma_{l'}\}$, where $f_\alpha \in U_{\xi}({\frak m}_-)$ are generators of the corresponding algebra $U_{\xi}({\frak m}_-)\subset U_{\xi}(\g)$ defined in Section \ref{WHITT}. Let $\chi:U_{\xi}({\frak m}_-)\rightarrow \mathbb{C}$ be any character defined in Proposition \ref{irrepod}.
Then any finite--dimensional $U_{\eta}({\frak g})$--module contains a non--zero Whittaker vector with respect to the subalgebra $U_{\eta}^{\widetilde{g}}({\frak m}_-)=\tilde{g}^{-1}U_{\xi}({\frak m}_-)$ and the character $\chi^{\widetilde{g}}$ given by the composition of $\chi$ and $\widetilde{g}$, $\chi^{\widetilde{g}}=\chi \circ \widetilde{g}:U_{\eta}^{\widetilde{g}}({\frak m}_-) \rightarrow \mathbb{C}$.

Moreover, ${\rm dim}~U_{\eta}^{\widetilde{g}}({\frak m}_-)=m^{\frac{1}{2}{\rm dim}~\mathcal{O}_{\pi\eta}}$, where $\mathcal{O}_{\pi\eta}$ is the conjugacy class of $\pi\eta \in G_\mathcal{C}$.
\end{theorem}

\begin{proof}
Let
\begin{equation*}
{\Delta}_0=\{\alpha\in \Delta|s(\alpha)=\alpha\},
\end{equation*}
and $\Gamma$ the set of simple roots in a system of positive roots $\Delta_+^{s^{-1}}$ defined in Proposition \ref{pord}. We shall need the parabolic subalgebra $\p$ of $\g$ and the parabolic subgroup $P$ associated to the subset $\Gamma_0=\Gamma\bigcap {\Delta}_{0}$ of simple roots and containing the Borel subalgebra corresponding to $\Delta_+^{s^{-1}}$. Let $\n$ and $\l$ be the nilradical and the Levi factor of $\p$, $N$ and $L$ the unipotent radical and the Levi factor of $P$, respectively.

Note that we have natural inclusions of Lie algebras $\p\supset\n$. We also denote by $\opn$ the nilpotent subalgebra opposite to $\n$.
Denote by $N$ the subgroup of $G$ corresponding to the Lie subalgebra $\n$ and by $\overline{N}$ the opposite unipotent subgroup in $G$ with the Lie algebra $\overline{\n}$.
Let $Z$ be the subgroup of $G$ generated by the semi--simple part of the Levi factor $L$ corresponding to the Lie subalgebra $\l$ and by the centralizer of $s$ in $H$. Denote by $\dot{s}$ a representative of $s$ in $G$. Let $N_s=\{ v \in N|\dot{s}v\dot{s}^{-1}\in \overline{N} \}$ and $H_0\subset H$ the subgroup corresponding to the orthogonal complement $\h_0$ of $\h'$ in $\h$ with respect to the Killing form.

By Theorem 5.2 in \cite{S10} for every $\mathcal{C}\in C(W)$ and $s\in \mathcal{C}$ there is a system of positive roots
$\Delta_+^{s^{-1}}$ as in Proposition \ref{pord} and such that all conjugacy classes in the stratum $G_\mathcal{C}$ intersect the variety $\dot{s}H_0N_s$ which is a subvariety of the transversal slice $\Sigma_s=\dot{s}ZN_s$ to the set of conjugacy classes in $G$. Note that $N_s$ is not a subgroup in $N_+$. But every element of $n_s\in N_s$ can be uniquely  factorized as follows $n_s=n_s^+n_s^-$, $n_s^\pm\in N_s\cap N_\pm$, $N_s\cap N_-=N\cap N_-$, $N_s\cap N_+\subset M_+$, where $M_+\subset N_+$ is the subgroup corresponding to the Lie subalgebra $\m_+$. Therefore every element $\dot{s}h_0n_s\in \dot{s}H_0N_s$ can be represented as follows $\dot{s}h_0n_s=\dot{s}h_0n_s^+n_s^-$, and conjugating by $n_s^-$ we obtain that $\dot{s}h_0n_s$ is conjugate to $n_s^-\dot{s}h_0n_s^+$.
Since the decomposition $s=s^1s^2$ is reduced 
\begin{equation}\label{inters}
\dot{s}^{-1}(N_s\cap N_-)\dot{s}=\dot{s}^{-1}(N\cap N_-)\dot{s}\subset M_+.
\end{equation} 
Taking into account that $H_0$ normalizes $M_+$ we have $n_s^-\dot{s}h_0n_s^+=\dot{s}h_0h_0^{-1}\dot{s}^{-1}n_s^-\dot{s}h_0n_s^+=\dot{s}h_0m_s$, $m_s=h_0^{-1}\dot{s}^{-1}n_s^-\dot{s}h_0n_s^+\in M_+$. We deduce that all conjugacy classes in the stratum $G_\mathcal{C}$ intersect the variety $\dot{s}H_0M_+$.
 
Recall that by Proposition 6.2 in \cite{S11} we have $u=z_+n'\dot{s}{m}$ for some $z_+\in Z\cap N_+, n'\in N, {m}\in M_+$, where
$u\in G$ is an element of the form
\begin{equation}\label{defu}
u=\prod_{i=1}^{l'}exp[t_{i} X_{-\gamma_i}],
\end{equation}
$t_i\in \mathbb{C}$ are non--zero constants depending on the choice of the representative $\dot{s}$ and the product over roots is taken in the order opposed to the normal order associated to $s^{-1}$.
Factorizing $n'=n_+n_-$, $n_\pm\in N\cap N_\pm$ and recalling (\ref{inters}) we obtain
$u=z_+n'\dot{s}{m}=z_+n_+\dot{s}\dot{s}^{-1}n_-\dot{s}{m}=z_+n_+\dot{s}m',m'=\dot{s}^{-1}n_-\dot{s}{m}\in M_+$, and hence 
$\dot{s}=n_+^{-1}z_+^{-1}u{m'}^{-1}$. 

Now let $\dot{s}h_0m_s,h_0\in H_0,m_s\in M_+$ be an element of $\dot{s}H_0M_+$. Using the previous expression for $\dot{s}$ we can write $\dot{s}h_0m_s=n_+^{-1}z_+^{-1}u{m'}^{-1}h_0m_s$. Conjugating this element by $z_+n_+$ and recalling that $H_0$ normalizes $M_+$ we infer that $\dot{s}h_0m_s$ is conjugate to 
$$
u{m'}^{-1}h_0m_sn_+^{-1}z_+^{-1}=uh_0n=\lambda_0(h_0^{\frac12}n,h_0^{-\frac12}u^{-1}),
$$
where $n=h_0^{-1}{m'}^{-1}h_0m_sn_+^{-1}z_+^{-1}\in N_+$, $h_0^{\frac12}\in H_0$ is any element such that $h_0^{\frac12}h_0^{\frac12}=h_0$, $\lambda_0$ is defined immediately after Proposition \ref{LO}. We deduce that all conjugacy classes in the stratum $G_\mathcal{C}$ intersect the variety $uH_0N_+$.

By part (iv) of Proposition \ref{qcoadj} we conclude that if $\eta \in {\rm Spec}(Z_0)$ satisfies $\pi\eta\in G_\mathcal{C}$ then there is a quantum coadjoint transformation $\widetilde{g}$ such that $\widetilde{\pi}(\widetilde{g}\eta)=(h_0^{\frac12}n,h_0^{-\frac12}u^{-1})$ for some $n\in N_+$, $h_0^{\frac12}\in H_0$.

Denote $\xi=\widetilde{g}\eta$. From the definition of the map $\widetilde{\pi}$ and of the element $u$ it follows that 
$$
\exp(\xi(x_{\beta_D}^-)X_{-\beta_D})\exp(\xi(x_{\beta_{D-1}}^-)X_{-\beta_{D-1}})\ldots \exp(\xi(x_{\beta_1}^-)X_{-\beta_1})=u
$$
which implies $\xi((X_{\gamma_i}^-)^m)=\frac{t_i}{(\varepsilon_{\gamma_i}-\varepsilon_{\gamma_i}^{-1})^m}\neq 0$ for $i=1,\ldots ,l'$ and that $\xi((X_\beta^-)^m)=0$ for $\beta\in \Delta_{\m_+}$, $\beta\not\in \{\gamma_1, \ldots ,\gamma_{l'}\}$. 

By the definition of the elements $f_\beta$ with $\beta=\sum_{i=1}^lm_i\alpha_i$ we have $f_\beta=\prod_{i,j=1}^lL_j^{m_in_{ij}}X_\beta^-$. Therefore the commutation relations between elements $L_j$ and $X_\beta^-$ imply that $f_\beta^m=c_\beta\prod_{i,j=1}^lL_j^{mm_in_{ij}}(X_\beta^-)^m$, where $c_\beta$ are non--zero constants, and hence $\xi(f_\beta^m)=c_\beta\prod_{i,j=1}^l\xi(L_j^m)^{m_in_{ij}}\xi((X_\beta^-)^m)$. Since $\xi(L_j)\neq 0$ for $j=1,\ldots ,l$, $\xi((X_{\gamma_i}^-)^m)=\frac{t_i}{(\varepsilon_{\gamma_i}-\varepsilon_{\gamma_i}^{-1})^m}\neq 0$ for $i=1,\ldots ,l'$ and $\xi((X_\beta^-)^m)=0$ for $\beta\in \Delta_{\m_+}$, $\beta\not\in \{\gamma_1, \ldots ,\gamma_{l'}\}$ we deduce 
$\xi(f_{\gamma_i}^m)=a_i\neq 0$ for $i=1,\ldots ,l'$ and $\xi(f_\beta^m)=0$  for $\beta\in \Delta_{\m_+}$, $\beta\not\in \{\gamma_1, \ldots ,\gamma_{l'}\}$. Thus $\xi$ satisfies the condition of Propositions \ref{irrepod} and \ref{Whitt}. Let $U_{\widetilde{g}\eta}({\frak m}_-)=U_{\xi}({\frak m}_-)$ be the corresponding subalgebra in $U_{\xi}(\g)$.

Note that by Theorem 5.2 in \cite{S10} for any $g\in G_\mathcal{C}$ we have
$$
{\rm dim}~Z_G(g)={\rm dim}~\Sigma_s,
$$
where $Z_G(g)$ is the centralizer of $g$ in $G$. 

By the definition of $\Sigma_s$ we also have ${\rm dim}~\Sigma_s=l(s)+2D_0+{\rm dim}~{\h'}^\perp$. 
Observe also that ${\rm dim}~G=2D+{\rm dim}~\h$ and ${\rm dim}~\h-{\rm dim}~{\h'}^\perp={\rm dim}~{\h'}=l'$,
and hence from (\ref{dimm}) we deduce that ${\rm dim}~\m_-=D-D_0-\frac12(l(s)-l')=\frac{1}{2}({\rm dim}~G-{\rm dim}~\Sigma_s)=\frac{1}{2}{\rm dim}~\mathcal{O}_g$ and ${\rm dim}~U_{\eta}^{\widetilde{g}}({\frak m}_-)={\rm dim}~U_{\widetilde{g}\eta}({\frak m}_-)=m^{{\rm dim}~\m_-}=m^{\frac{1}{2}{\rm dim}~\mathcal{O}_g}$, where $\mathcal{O}_g$ is the conjugacy class of any $g \in G_\mathcal{C}$.

In particular, ${\rm dim}~U_{\eta}^{\widetilde{g}}({\frak m}_-)=m^{\frac{1}{2}{\rm dim}~\mathcal{O}_{\pi\eta}}$, where $\mathcal{O}_{\pi\eta}$ is the conjugacy class of $\pi\eta \in G_\mathcal{C}$.

The remaining statements of this theorem are consequences of Proposition \ref{Whitt}.

\end{proof}

For given $\eta\in {\rm Spec}(Z_0)$ and $\widetilde{g}\in \mathcal{G}$ as in Theorem \ref{DKPconj} we denote $\mathbb{C}_{\chi^{\widetilde{g}}}$ the corresponding one--dimensional representation of $U_{\eta}^{\widetilde{g}}({\frak m}_-)$. Let $Q_{\chi^{\widetilde{g}}}$ be the induced left $U_{\eta}(\g)$--module, $Q_{\chi^{\widetilde{g}}}=U_{\eta}(\g)\otimes_{U_{\eta}^{\widetilde{g}}({\frak m}_-)}\mathbb{C}_{\chi^{\widetilde{g}}}$. Let $W^s_{\varepsilon,\eta}(G)={\rm End}_{U_{\eta}(\g)}(Q_{\chi^{\widetilde{g}}})^{opp}$ be the algebra of $U_{\eta}(\g)$--endomorphisms of $Q_{\chi^{\widetilde{g}}}$ with the opposite multiplication. The algebra $W^s_{\varepsilon,\eta}(G)$ is called a q-W algebra associated to $s\in W$. Denote by $U_{\eta}(\g)-{\rm mod}$ the category of finite--dimensional left $U_{\eta}(\g)$--modules and by $W^s_{\varepsilon,\eta}(G)-{\rm mod}$ the category of finite--dimensional left $W^s_{\varepsilon,\eta}(G)$--modules. Observe that if $V\in U_{\eta}(\g)-{\rm mod}$ then the algebra $W^s_{\varepsilon,\eta}(G)$ naturally acts on the finite--dimensional space $V_{\chi^{\widetilde{g}}}={\rm Hom}_{U_{\eta}^{\widetilde{g}}({\frak m}_-)}(\mathbb{C}_{\chi^{\widetilde{g}}},V)={\rm Hom}_{U_{\eta}(\g)}(Q_{\chi^{\widetilde{g}}},V)$ by compositions of homomorphisms.

\begin{proposition}
Let $\Phi:E\mapsto Q_{\chi^{\widetilde{g}}}\otimes_{W^s_{\varepsilon,\eta}(G)}E$ be the functor from the category of finite--dimensional left $W^s_{\varepsilon,\eta}(G)$--modules to the category $U_{\eta}(\g)-{\rm mod}$. Its right adjoint  functor is given by $\Psi:V\mapsto V_{\chi^{\widetilde{g}}}$ and satisfies $\Psi \circ \Phi={\rm Id}$.
\end{proposition}

\begin{proof}
The first statement follows from the definitions.

For the second one, let $E$ be a finite--dimensional $W^s_{\varepsilon,\eta}(G)$--module. First we observe that by the definition of the algebra $W^s_{\varepsilon,\eta}(G)$ we have $W^s_{\varepsilon,\eta}(G)={\rm End}_{U_{\eta}(\g)}(Q_{\chi^{\widetilde{g}}})^{opp}={\rm Hom}_{U_{\eta}^{\widetilde{g}}({\frak m}_-)}(\mathbb{C}_{\chi^{\widetilde{g}}},Q_{\chi^{\widetilde{g}}})=(Q_{\chi^{\widetilde{g}}})_{\chi^{\widetilde{g}}}$ as a linear space, and hence
$(Q_{\chi^{\widetilde{g}}}\otimes_{W^s_{\varepsilon,\eta}(G)}E)_{\chi^{\widetilde{g}}}=E$. This proves the second statement.
\end{proof}

\end{document}